\documentclass[12pt,a4paper,leqno]{article}

\usepackage[utf8]{inputenc}
\usepackage[T1]{fontenc}
\usepackage[english]{babel}
\usepackage{amsthm}
\usepackage{amscd}
\usepackage{amsfonts}         
\usepackage{amsmath}
\usepackage{amssymb}
\usepackage{hyperref}
\usepackage{enumerate}
\usepackage{tikz}
\usepackage{tikz-cd}
\usepackage{mathrsfs}  

\newcommand{\fref}[1]{\hyperref[{#1}]{\ref*{#1}}}

\newcommand{\Ab}{\mathbb{A}}
\newcommand{\Bb}{\mathbb{B}}

\newcommand{\Hb}{\mathbb{H}}

\newcommand{\Lb}{\mathbb{L}}

\newcommand{\Pb}{\mathbb{P}}
\newcommand{\Qb}{\mathbb{Q}}

\newcommand{\Zb}{\mathbb{Z}}

\newcommand{\Ac}{\mathcal{A}}

\newcommand{\Ec}{\mathcal{E}}
\newcommand{\Fc}{\mathcal{F}}

\newcommand{\Mc}{\mathcal{M}}
\newcommand{\Nc}{\mathcal{N}}
\newcommand{\Oc}{\mathcal{O}}

\newcommand{\Rc}{\mathcal{R}}
\newcommand{\Sc}{\mathcal{S}}

\newcommand{\Ls}{\mathscr{L}}
\newcommand{\Ms}{\mathscr{M}}

\newcommand{\Ps}{\mathscr{P}}

\newcommand{\PCob}{{\underline\Omega}}

\newcommand{\op}{\mathrm{op}}

\newcommand{\coim}{\mathrm{\coim}}

\newcommand{\Hom}{\mathrm{Hom}}
\newcommand{\Spec}{\mathrm{Spec}}

\newcommand{\Fl}{\mathrm{Fl}}
\newcommand{\bl}{\mathrm{Bl}}
\newcommand{\Td}{\mathrm{Td}}
\newcommand{\wtil}{\widetilde}

\newcommand{\colim}{\mathrm{colim}}

\newtheorem{theo}{Tplottin ubuntuheorem}[section]
\theoremstyle{plain}
\newtheorem{thm}[theo]{Theorem}
\newtheorem{lem}[theo]{Lemma}
\newtheorem{prop}[theo]{Proposition}
\newtheorem{cor}[theo]{Corollary}
\newtheorem*{thm*}{Theorem}
\newtheorem*{lem*}{Lemma}
\newtheorem*{prop*}{Proposition}
\newtheorem*{cor*}{Corollary}

\theoremstyle{definition}
\newtheorem{defn}[theo]{Definition}

\newtheorem{ex}[theo]{Example}

\theoremstyle{remark}
\newtheorem{rem}[theo]{Remark}
\newtheorem{cons}[theo]{Construction}
\newtheorem{war}[theo]{Warning}

\title{Chern classes in precobordism theories}
\author{Toni Annala}

\newcommand{\Addresses}{{
  \bigskip
  \footnotesize

  Toni Annala, \textsc{Department of Mathematics, University of British Columbia,
    Vancouver, BC V6T1Z2 Canada}\par\nopagebreak
  \textit{E-mail address:} \texttt{tannala@math.ubc.ca}

}}

\date{}
\begin{document}

\maketitle

\begin{abstract}
We construct Chern classes of vector bundles in the universal precobordism theory of Annala--Yokura over an arbitrary Noetherian base ring of finite Krull dimension. As an immediate corollary of this, we show that the Grothendieck ring of vector bundles can be recovered from the universal precobordism ring, and that we can construct candidates for Chow rings satisfying an analogue of the classical Grothendieck--Riemann--Roch theorem. We also strengthen the weak projective bundle formula of Annala--Yokura to work for arbitrary projective bundles.
\end{abstract}

\tableofcontents

\section{Introduction}

The theory of \emph{algebraic cobordism} is supposed to be the finest possible cohomology theory in algebraic geometry satisfying certain restrictions, but constructing such a theory in general has proven to be a hard problem. The importance of this problem can be highlighted by the following observation: a satisfactory theory of algebraic cobordism rings will also yield a satisfactory theory of Chow rings. This means that if we can understand, for example, how to define algebraic cobordism for singular schemes, then we should also understand how to define the Chow rings for singular schemes. 

The purpose of this work is to pursue one approach, based on derived algebraic geometry, that should yield candidates of algebraic cobordism in great generality. Namely, given a Noetherian base ring $A$ of finite Krull dimension, we consider the \emph{universal precobordism rings} $\PCob^*(X)$ for quasi-projective derived $A$-schemes $X$ introduced in \cite{AY}. We show that these rings satisfy many properties expected from a good theory of algebraic cobordism.

\subsection*{Summary of results}
We start by Section \fref{FGLSect}, which deals with the formal group law of the universal precobordism theory. The main result of Section \fref{FGLSubSect} is Theorem \fref{FGLTheorem} showing that there is a formal group law $F(x,y)$ so that given line bundles $\Ls_1$ and $\Ls_2$ on $X$, we get the equation
$$e(\Ls_1 \otimes \Ls_2) = F\bigl(e(\Ls_1), e(\Ls_2) \bigr) \in \PCob^1(X)$$
on Euler classes. This strengthens the analogous result of \cite{AY} where the line bundles were required to be globally generated. This result allows us to give in Theorem \fref{UniversalPropOfPreCob} a simple universal property for $\PCob^*$: it is the \emph{universal oriented cohomology theory} (see Definition \fref{OrientedCohomDef}). The definition of an oriented cohomology theory is a rather simple one: in addition to desirable functorial properties, the only additional restrictions concern Euler classes of line bundles.

Section \fref{ChernClassConsSect} is the technical heart of the paper. In Section \fref{ChernClassConsSubSect} we construct \emph{Chern classes}
$$c_i(E) \in \PCob^i(X)$$
of a vector bundle $E$ on $X$. This is done by carefully applying derived blow ups of Khan--Rydh \cite{Khan} to the derived schemes of interest (Construction \fref{ChernClassCons}). In Section \fref{ChernPropertiesSubSect} we begin the study of basic properties of Chern classes. The first main result of this section is Theorem \fref{ChernProperties}, which shows that Chern classes are natural in pullbacks and that the top Chern class coincides with the Euler class. The second is Theorem \fref{ProjectiveBundleChernClass}, which shows that, given a rank $r$ vector bundle $E$ on $X$,
$$\sum_{i=0}^r (-1)^i c_{r-i}(E^\vee) \bullet c_1(\Ls)^i = 0 \in \PCob^*(\Pb(E)).$$
This formula will play an important role in Section \fref{ApplicationsSect}. We then prove the validity of the \emph{splitting principle} in Section \fref{SplittingPrincipleSubSect} (Theorem \fref{SplittingPrinciple}), and as easy corollaries we can deduce Theorems \fref{NilpChernClass} and \fref{WSF} showing respectively that Chern classes are nilpotent and that Whitney sum formula holds.  

In Section \fref{ApplicationsSect} we focus on applications of the results obtained in the previous sections. We begin with Section \fref{CFSubSect}, whose main result is Theorem \fref{GeneralCF}: if $\Zb_m$ is the integers considered as an $\Lb$-algebra via the map that classifies the multiplicative formal group law, then we get a natural isomorphism of rings
$$\PCob^*(X) \otimes_\Lb \Zb_m \cong K^0(X).$$
In characteristic 0, this is a slight generalization of the corresponding result in \cite{An} (since we have fewer relations), but over a more general ground ring $A$, the result is completely new. In Section \fref{RRSubSect}, we construct candidates for Chow rings by the formula
$$\PCob^*_a(X) := \Zb_a \otimes_\Lb \PCob^*(X),$$
where $\Zb_a$ is the integers, again considered as an $\Lb$-algebra, but this time via the map classifying the additive formal group law. In Theorem \fref{GeneralRR}, we show that this new theory $\PCob^*_a$ satisfies an analogue of the classical Grothendieck--Riemann--Roch theorem.

The main result of Section \fref{PBFSubSect} --- Theorem \fref{PBF} --- shows that if $E$ is a vector bundle of rank $r$ on $X$, then we have a natural isomorphism
$$\PCob^*(\Pb(E)) \cong \PCob^*(X)[t] / \langle f(t) \rangle,$$
where $t = c_1(\Oc(1))$ and
$$f(t) := \sum_{i=0}^r (-1)^i c_{r-i}(E^\vee) \bullet c_1(\Ls)^i.$$
The proof uses in an essential way the computation of precobordism with line bundles $\PCob^{*,1}$ from \cite{AY}. Note that the result strengthens both Theorem \fref{ProjectiveBundleChernClass} and Corollary 6.24 of \cite{AY}. We then end by quickly showing that the arguments generalize to arbitrary bivariant precobordism theories $\Bb^*$ (Theorem \fref{BivariantPBF}).

\subsection*{Related work}

Algebraic cobordism was originally introduced by Voevodsky in his proof of the Milnor conjecture. In order to understan the theory better, Levine and Morel in their foundational work (see \cite{LM}) gave a geometric construction for the algebraic cobordism rings $\Omega^*(X)$ of smooth varieties $X$ over a field of characteristic 0. In \cite{An}, the author extended the construction of the cobordism rings to arbitrary quasi-projective derived schemes under the same characteristic 0 assumption. Moreover, the cobordism rings were also shown to be a part of a larger \emph{bivariant theory}, still denoted by $\Omega^*$. In this paper, we will largely ignore all the bivariant theoretic aspects, since it would require us to introduce the bivariant formalism.

This paper builds upon the earlier paper \cite{AY} where precobordism theories over a Noetherian base ring of finite Krull dimension were first introduced. For us, the most important results from this work are Theorems 6.12 and 6.13, expressing the structure of the bivariant precobordism with line bundles $\Bb^{*,1}$ in terms of the original precobordism theory $\Bb^*$. Also the variants $\Bb^{*,1}_\mathrm{gl}$ and $\Bb^*_{\Pb^\infty},$ and their relationship with $\Bb^{*,1}$ will play an important role. 

There has been considerable interest of studying cobordism theories coming from motivic homotopy theory, see for example \cite{De, EHKY}. At this moment, their relationship with the work pursued here remains poorly understood. Note that since the theories constructed do not have homotopy invariance in general (since $K^0$ does not), the two approaches can not give the same result in general. Investigating possible connections is left for future work.

\subsection*{Acknowledgements}
The author would thank his advisor Kalle Karu for multiple discussions. The author was supported by the Vilho, Yrj\"o and Kalle V\"ais\"al\"a Foundation of the Finnish Academy of Science and Letters. 

\subsection*{Conventions}
Our derived $A$-schemes are locally modeled by simplicial $A$-algebras. If a diagram is said to commute, it is to be understood that it commutes coherently up to homotopy. An \emph{inclusion} or \emph{embedding} of vector bundles $E \hookrightarrow F$ is a morphism whose dual $F^\vee \to E^\vee$ is surjective (which is to say, it is surjective on $\pi_0$).

\section{Background}

In this section, we are going to recall definitions and results that are necessary for the main part of the paper. As a rule, only results that are new are proven.

\subsection{Projective bundles}

Let $X$ be a derived scheme and let $E$ be a vector bundle on $X$. We define the \emph{projective bundle} $\pi: \Pb_X(E) \to X$ as the derived $X$-scheme so that given an $X$-scheme $Y$, the space of morphisms $Y \to \Pb_X(E)$ is canonically equivalent to the space ($\infty$-groupoid) of surjections
$$E^\vee \vert_Y \to \Ls^\vee$$
where $\Ls$ is a line bundle on $Y$. It is known that $\Pb_X(E)$ exists and that the structure morphism $\pi$ is smooth. Moreover, the universal property induces a canonical surjection $E^\vee \to \Oc_{\Pb_X(E)}(1)$ to the \emph{anticanonical line bundle} of $\Pb_X(E)$. The derived pushforward to $X$ of the canonical surjection is equivalent to the identity morphism of $E$. If it should cause no confusion, we will often denote the projective bundle simply by $\Pb(E)$ and the anticanonical line bundle by $\Oc(1)$.

Note that the whole discussion of the last paragraph can be dualized so that $\Pb(E)$ represents embeddings $\Ls \hookrightarrow E \vert_Y$, and so that there is a canonical embedding $\Oc(-1) \hookrightarrow E$ on $\Pb(E)$. It is then easy to see that given an embedding $E \hookrightarrow F$ of vector bundles on $X$, we get the induced \emph{linear embedding} $i: \Pb(E) \hookrightarrow \Pb(F)$. It is known that $i$ is a quasi-smooth embedding of virtual codimension $\mathrm{rank}(F) - \mathrm{rank}(E)$. For the details, the reader may consult Section 2.8 of \cite{AY}.

\begin{war}
Note that there are two conventions for projective bundles so that what we call here $\Pb(E)$, might somewhere else be called $\Pb(E^\vee)$. Our convention is closer to the one standard in intersection theory.
\end{war}

\subsection{Derived vanishing loci}

Given a global section $s$ of a vector bundle $E$ of rank $r$ on a derived scheme $X$, we can form the $X$-scheme called \emph{derived vanishing locus} $i: V(s) \hookrightarrow X$ of $s$ in $X$ as the homotopy Cartesian product
\begin{equation}
\begin{tikzcd}
V(s) \arrow[->]{r}{i} \arrow[->]{d}{} & X \arrow[]{d}{} \\
X \arrow[->]{r}{0} & E
\end{tikzcd}
\end{equation}

$$
\CD
V(s) @>{i}>> X \\
@VVV @V{s}VV \\
X @>{0}>> E
\endCD
$$ 
Note that $i$ is a quasi-smooth embedding of virtual codimension $r$. By the defining property of homotopy Cartesian products, $V(s)/X$ has another universal property: namely, the space of morphisms over $X$ from $Y$ to $V(s)$ is canonically equivalent to the space of paths $s \vert_Y \sim 0$ in the space of global sections $\Gamma(Y; E \vert_Y)$. This allows us to make the following observation, recorded as a lemma.

\begin{lem}\label{PreferredPathLemma} In the derived vanishing locus $i: V(s) \hookrightarrow X$, there is a canonical path $\alpha: i^* s \sim 0$, which is natural under pullbacks of the data $(E,s)$ in the obvious way.
\end{lem}
\begin{proof}
Indeed $\alpha$ is the path corresponding to the identity morphism $V(s) \to V(s)$ over $X$. The other claims follow trivially.
\end{proof}

Our next goal is to prove a kind of a transitivity result for derived vanishing loci (Proposition \fref{VanishingLociInExtensions}). 

\begin{cons}\label{SectionLift}
Suppose now $X$ is a derived scheme, and 
\begin{equation}\label{ExactSeq1}
0 \to E' \to E \to E'' \to 0
\end{equation}
is a short exact sequence of vector bundles. If $s$ is a global section of $E$, then it maps to a global section $s''$ of $E'$; denote the derived vanishing locus of $s''$ by $Z_1$. The pullback of (\fref{ExactSeq1}) induces a (homotopy) fibre sequence 
\begin{equation}\label{GSFib}
\Gamma(Z_1; E') \to \Gamma(Z_1; E) \to \Gamma(Z_1; E'')
\end{equation}
of spaces of global sections. The natural path $\alpha: s''\vert_{Z_1} \sim 0$ in $\Gamma(Z_1; E'')$ given by Lemma \fref{PreferredPathLemma} allows us to lift $s$ to a natural element (unique up to homotopy) $s' \in \Gamma(Z_1; E')$; denote the derived vanishing locus of $s'$ by $Z_2$. The following proposition should be thought as some kind of transitivity result for derived vanishing loci. Notice how this construction is natural under pullbacks of the data $(E, E', E'', s)$
\end{cons}

\begin{prop}\label{VanishingLociInExtensions}
Let $X, s$ and $Z_2$ be as in Construction \fref{SectionLift}. Then there exists an equivalence
$$Z_2 \simeq Z,$$
where $Z$ is the derived vanishing locus of $s$ in $X$, which is natural under pullbacks of the data $(E, E', E'', s)$
\end{prop}
\begin{proof}
In order to get a morphism $Z_2 \to Z$, we need to find a path $s \vert_{Z_2} \sim 0$ in $\Gamma(Z_2; E)$. But this is easy: there is a natural path $\alpha: s'' \vert_{Z_2} \sim 0$ in $\Gamma(Z_2; E'')$ which lifts to a natural path $\wtil \alpha: s\vert_{Z_2} \sim s'\vert_{Z_2}$ (this is essentially how $s'$ was constructed). Recalling that $Z_2$ was defined as the vanishing locus of $s'$, we obtain a natural path $\beta: s'\vert_{Z_2} \sim 0$ in $\Gamma(Z_2; E')$, which maps to give a natural path $s'\vert_{Z_2} \sim 0$ in $\Gamma(Z_2; E)$ denoted abusively by $\beta$. The desired natural path $s\vert_{Z_2} \sim 0$ is then given by the composition of $\wtil \alpha$ and $\beta$, and we get a natural map $\psi: Z_2 \to Z$.

We are left to show that $\psi$ is an equivalence, which can checked locally on $X$. We can therefore reduce to the case where all the vector bundles are trivial, and that the sequence (\fref{ExactSeq1}) is the standard split exact sequence 
$$\Oc_X^{\oplus n} \to \Oc_X^{\oplus n + m} \to \Oc_X^{\oplus m}.$$
In this situation $s$ corresponds to an $(n+m)$-tuple $(s_1,...,s_{n+m})$ of functions of $X$, $s''$ corresponds to $(s_{n+1},...,s_{n+m})$ and $s'$ corresponds to $(s_1,...,s_{n})$, and the equivalence follows from the equivalence of the derived intersection $V(s_1,...,s_n) \cap V(s_{n+1},...,s_{n+m})$ with $V(s_1,...,s_{n+m})$.
\end{proof}

The following result shows that all linear embeddings of projective bundles are vanishing loci of sections of vector bundles.

\begin{prop}\label{LinearEmbeddingIsVanishingLocus}
Let $X$ be a derived scheme, and let 
\begin{equation}\label{LEVLSeq}
0 \to E' \xrightarrow{i} E \xrightarrow{f} E'' \to 0
\end{equation}
be an exact sequence of vector bundles on $X$. Consider the vector bundle $E(1)$ on $\Pb(E)$ and the section $s \in \Gamma(\Pb(E); E(1))$ corresponding via the natural identifications 
$$\Gamma(\Pb(E); E(1)) \simeq \Gamma(X; E^\vee \otimes E) \simeq \Hom_X(E,E)$$ to the identity morphism. Then the derived vanishing locus of $f_*(s) \in \Gamma(\Pb(E); E''(1))$ is equivalent to the linear embedding $\Pb(E') \hookrightarrow \Pb(E)$ over $\Pb(E)$.
\end{prop}
\begin{proof}
Consider the diagram
\begin{equation*}
\begin{tikzcd}
\Gamma(\Pb(E), E(1)) \arrow[->]{r}{f_*} \arrow[->]{d}{\simeq} & \Gamma(\Pb(E), E''(1)) \arrow[->]{r}{i^*} \arrow[->]{d}{\simeq} & \Gamma(\Pb(E'), E''(1)) \arrow[->]{d}{\simeq} \\
\Gamma(X, E^\vee \otimes E) \arrow[->]{r}{(1 \otimes f)_*} \arrow[->]{d}{\simeq} & \Gamma(X, E^\vee \otimes E'') \arrow[->]{r}{(i^\vee \otimes 1)_*} \arrow[->]{d}{\simeq} & \Gamma(X, E'^\vee \otimes E'') \arrow[->]{d}{\simeq} \\
\Hom_X(E, E) \arrow[->]{r}{f \circ} & \Hom_X(E, E'') \arrow[->]{r}{\circ i} & \Hom_X(E', E'').
\end{tikzcd}
\end{equation*}
It is clear that the two squares on the left hand side commute. To see why the upper right square must commute, we recall that the linear embedding $i: \Pb(E') \hookrightarrow \Pb(E)$ is induced by the surjection
$$E^\vee \xrightarrow{i^\vee} E'^\vee \to \Oc(1)$$
on $\Pb(E)$, proving the commutativity of 
\begin{equation*}
\begin{tikzcd}
\Gamma(\Pb(E); \Oc(1)) \arrow[->]{r}{i^*} \arrow[->]{d}{\simeq} & \Gamma(\Pb(E'); \Oc(1)) \arrow[->]{d}{\simeq} \\
\Gamma(X; E^\vee) \arrow[->]{r}{i^\vee_*} & \Gamma(X; E'^\vee).
\end{tikzcd}
\end{equation*}
The commutativity of the bottom right square follows from the dual statement that
\begin{equation*}
\begin{tikzcd}
\Gamma(X, E^\vee \otimes E'') \arrow[->]{r}{(i^\vee \otimes 1)_*} \arrow[->]{d}{\simeq} & \Gamma(X, E'^\vee \otimes E'') \arrow[->]{d}{\simeq} \\
\Hom_X(E''^\vee, E^\vee) \arrow[->]{r}{i^\vee \circ} & \Hom_X(E''^\vee, E'^\vee)
\end{tikzcd}
\end{equation*}
commutes.

But now it is clear that there exists a path $i^*(f_*(s)) \sim 0$ inside $\Gamma(\Pb(E'); E''(1))$ since it corresponds to the morphism $f \circ i$, and therefore has a preferred nullhomotopy by exactness of (\fref{LinearEmbeddingIsVanishingLocus}). Therefore we get a morphism $\Pb(E') \to V(i^*(f_*(s)))$ over $\Pb(E)$. Note that both derived schemes have the same virtual codimension in $\Pb(E)$. Therefore one can conclude, after checking locally that $V(i^*(f_*(s)))$ is smooth over $X$ (a standard argument we are not going to repeat here), that the map $\Pb(E') \to V(i^*(f_*(s)))$ must be an equivalence as a quasi-smooth embedding of virtual codimension 0 (recall that an $X$-morphism between smooth $X$-schemes is quasi-smooth).
\end{proof}

\subsection{Derived blow ups}

Let us recall the definition of a derived blow-up from \cite{Khan} (Section 4.1). Namely, given a derived scheme $X$ and a quasi-smooth embedding $j: Z \hookrightarrow X$, the \emph{derived blow up of $X$ along $Z$} is the derived scheme $\pi: \bl_Z(X) \to X$ so that, given another $X$-scheme $\pi_Y: Y \to X$, the space of morphisms $Y \to \bl_Z(X)$ over $X$ is canonically equivalent to the space of commuting diagrams
\begin{center}
\begin{tikzcd}
D \arrow[hookrightarrow]{r}{i_D} \arrow[->]{d}{g} & Y \arrow[->]{d}{\pi_Y} \\ 
Z \arrow[hookrightarrow]{r}{j} & X
\end{tikzcd}
\end{center}
so that
\begin{enumerate}
\item $i_D$ is a quasi-smooth embedding of virtual codimension 1;
\item the above square truncates to a Cartesian square of schemes; 
\item the canonical morphism $g^* \Nc^\vee_{Z/X} \to \Nc^\vee_{D/Y}$ of conormal bundles is surjective.
\end{enumerate}

\noindent Let us recall some of the basic properties of blow ups that are going to be useful for us.

\begin{thm}[cf. Theorem 4.1.5 \cite{Khan}]
Let $Z \hookrightarrow X$ be a quasi-smooth immersion. Then
\begin{enumerate}
\item the blow up $\pi: \bl_Z(X) \to X$ is natural in pullbacks;

\item if $X \hookrightarrow Y$ is a quasi-smooth immersion, then there exists a canonical quasi-smooth immersion $\bl_Z(X) \hookrightarrow \bl_Z(Y)$ called the \emph{strict transform};

\item if both $Z$ and $X$ are classical, $\pi: \bl_Z(X) \to X$ is the classical blow up of $X$ along $Z$.
\end{enumerate}
\end{thm}

The following proposition gives an explicit presentation for the derived blow up of $X$ at $Z$ in the case where $Z$ is the derived vanishing locus of a global section $s$ of a vector bundle $E$.

\begin{prop}\label{ProjectiveBlowUp}
Let $X$ be a derived scheme, $E$ be a vector bundle on $X$ and $s$ a global section of $E$. Consider the natural exact sequence
\begin{equation}\label{ProjCanonicalSeq}
0 \to \Oc(-1) \to E \to Q \to 0
\end{equation}
of vector bundles on $\Pb(E)$, and denote by $s''$ the image of $s$ under the composition
$$\Gamma(X; E) \simeq \Gamma(\Pb(E); E) \to \Gamma(\Pb(E); Q).$$
If we denote by $Z$ the derived vanishing locus of $s$ in $X$ and by $Y$ the derived vanishing locus of $s''$ in $\Pb(E)$, then there is an equivalence
$$Y \simeq \bl_Z(X)$$
of derived schemes which is natural under pullbacks of the data $(E,s)$.
\end{prop}
\begin{proof}
Our first task is to find a natural map $Y \to \bl_Z(X)$. By composing with the projection $\pi: \Pb(E) \to X$, we get a natural map $Y \to X$. Moreover, the exact sequence (\fref{ProjCanonicalSeq}) yields as in Construction \fref{SectionLift} a natural global section $s'$ of $\Oc_Y(-1)$ whose vanishing locus $D \hookrightarrow Y$ is naturally identified with the pullback of $Z \hookrightarrow X$ via $\pi$. Therefore we get a homotopy commutative square
\begin{equation}\label{FundamentalSquare}
\begin{tikzcd}
D \arrow[hookrightarrow]{r}{} \arrow[->]{d}{\pi\vert_Z} & Y \arrow[->]{d} \\ 
Z \arrow[hookrightarrow]{r} & X
\end{tikzcd}
\end{equation}
which comes from a diagram
\begin{center}
\begin{tikzcd}
D \arrow[hookrightarrow]{r}{i_D} \arrow[->]{d}{\pi\vert_Z} & Y  \arrow[hookrightarrow]{r} & \Pb(E) \arrow[->]{d} \\ 
Z \arrow[hookrightarrow]{rr} & & X
\end{tikzcd}
\end{center}
whose outer square is homotopy Cartesian. We therefore deduce that (\fref{FundamentalSquare}) 
\begin{enumerate}
\item truncates into a Cartesian square of classical schemes;
\item the induced map $\pi \vert_Z^* \Nc^\vee_{Z/X} \to \Nc^\vee_{D/Y}$ on conormal bundles is surjective
\end{enumerate}
so that by the universal property of derived blow ups (see \cite{Khan}), we obtain a map $\phi: Y \to \bl_Z(X)$. As the whole construction was natural in derived pullbacks, also the morphism $\phi$ is.

We prove that $\phi$ is an equivalence in the usual way: by naturality, we can check this locally, and therefore we can use naturality again to reduce to the situation where we blow up the section $s = x_1 e_1 + x_2 e_2 + \cdots + x_n e_n$ of the trivial vector bundle $\Oc_{\Ab^n}^{\oplus n}$ on $\Ab^n = \Spec(\Zb[x_1,...,x_n])$. Note that the sequence (\fref{ProjCanonicalSeq}) is equivalent to the twisted Euler sequence 
$$0 \to \Oc(-1) \to \Oc^{\oplus n} \to T_{\Pb^{n-1}_{\Ab^n} / \Ab^n} (-1) \to 0$$
on $\Pb^{n-1}_{\Ab^n}$. Denoting by $y_i$ the dual basis for $e_i$, we see that the section $s''$ corresponds to $x_1 \partial_{y_1} + x_2 \partial_{y_2} + \cdots + x_n \partial_{y_n}$, whose vanishing locus consists clearly of exactly those points $\bigl((x_1,...,x_n), [y_1: ...: y_n] \bigr)$ so that $[y_1:...:y_n] = [x_1:...:x_n]$ whenever the latter is well defined and of arbitrary $\bigl((0,...,0),[y_1:...:y_n] \bigr)$ otherwise. Hence the vanishing locus is just the blow up of $\Ab^n$ at the origin and we are done.
\end{proof}

\begin{ex}\label{ProjBlowUpEx}
It might be enlightening (but it will certainly be useful later) to consider the special case of Proposition \fref{ProjectiveBlowUp} where the section $s$ vanishes nowhere. Since blowing up along an empty center doesn't do anything, we obtain the commutative triangle
\begin{center}
\begin{tikzcd}
X \arrow[->]{rd}{\mathrm{Id}} \arrow[hookrightarrow]{r}{i} & \Pb(E) \arrow[->]{d} \\ 
 & X
\end{tikzcd}
\end{center}
where $i$ identifies $X$ as the vanishing locus of $s' \in \Gamma(\Pb(E); Q)$. We claim that
\begin{equation}\label{TrivialBlowUpExSeq}
0 \to i^*\Oc_{\Pb(E)} \xrightarrow{i^*s} i^*E \to i^*Q \to 0
\end{equation}
is exact. Indeed, $i^* s$ is an embedding of vector bundles since $s$ was assumed not to vanish anywhere, and since the composition $i^* \Oc_{\Pb(E)} \to i^* Q$ is canonically homotopic to 0 by Lemma \fref{PreferredPathLemma}, the exactness follows from rank considerations. As the $i$-pullback of the canonical exact sequence (\fref{ProjCanonicalSeq}) is also exact, and since it shares the same right side as (\fref{TrivialBlowUpExSeq}), it follows from the universal property of $\Pb(E)$ that $i$ is equivalent to the linear embedding $s: \Pb(\Oc_X) \to \Pb(E)$.

Of course, if $s$ has a non-empty vanishing locus, then the linear embedding $s: \Pb(\Oc_X) \to \Pb(E)$ does not make sense. In that situation, we will have to first modify our space by taking the blow up, and then, on the modified space, there exists a canonical inclusion $\Oc(\Ec) \to E$ that allows us to define a map to $\Pb(E)$. In fact, one can show that blowing up is the ``most economic way'' of making $s$ equivalent to $s_1 \otimes s_2$, where $s_2: \Ls \hookrightarrow E$ an embedding and $s_1$ is a section of $\Ls$. 
\end{ex}

Next we are going to show that a linear embedding $\Pb(E') \hookrightarrow \Pb(E' \oplus E'')$ is almost a retraction. Let us begin with more general considerations: suppose $X$ is a derived scheme and
$$0 \to E' \to E \to E'' \to 0$$
an exact sequence of vector bundles on $X$. As $\Pb(E')$ is the derived vanishing locus of a section $s''$ of $E''(1)$ on $\Pb(E)$, Proposition  \fref{ProjectiveBlowUp} yields a natural surjection $E''^\vee (-1) \to \Oc(-\Ec)$ of vector bundles on $\bl_{\Pb(E')}(\Pb(E))$. We can then twist this surjection to obtain a surjection
\begin{equation}\label{TwistedSurj}
E''^\vee \to \Oc(1 -\Ec)
\end{equation}
then gives rise to a morphism $\rho: \bl_{\Pb(E')}(\Pb(E)) \to \Pb(E'')$.

\begin{prop}\label{BlowUpOfLinearEmbedding}
Let $X$ be a derived scheme, and let 
$$0 \to E' \to E' \oplus E'' \to E'' \to 0$$
be a split short exact sequence of vector bundles on $X$. Then the morphism $\rho$, constructed as above, expresses  $\Pb(E'')$ as a retract (up to homotopy) of $\bl_{\Pb(E')}(\Pb(E' \oplus E''))$.
\end{prop}
\begin{rem}
Of course, the last claim should also hold in the $\infty$-categorical sense and not only up to homotopy. But we have decided to restrict the generality in order to get away with a simpler proof.
\end{rem}
\begin{proof}
The only thing that is not obvious is that $\rho$ expresses $\Pb(E'')$ as a retract of the blow up. First of all, since $\Pb(E')$ and $\Pb(E'')$ do not meet inside $\Pb(E' \oplus E'')$, we can form the homotopy Cartesian square
\begin{equation*}
\begin{tikzcd}
\Pb(E'') \arrow[->]{r}{i'} \arrow[->]{d}{\mathrm{Id}} & \bl_{\Pb(E')}(\Pb(E' \oplus E'')) \arrow[->]{d}{} \\
\Pb(E'') \arrow[->]{r}{i} & \Pb(E' \oplus E'')
\end{tikzcd}
\end{equation*}
providing the map $i'$, which we claim to satisfy $\rho \circ i \simeq \mathrm{Id}$.

Note that we have to show that the $i'$-pullback of the surjection (\fref{TwistedSurj}) is equivalent to the canonical surjection $E''^\vee \to \Oc(1)$. But since
\begin{enumerate}
\item $i$ factors through the open subset $U := \Pb(E' \oplus E'') \backslash \Pb(E')$; 
\item the restriction $E''^\vee \to \Oc_U(1)$ of (\fref{TwistedSurj}) to $U$ corresponds to $s''^\vee \vert_U: E''^\vee(-1) \to \Oc_U$ by Example \fref{ProjBlowUpEx}
\end{enumerate} 
we can conclude that the pullback of (\fref{TwistedSurj}) corresponds to the global section $i^*(s''^\vee)$ of $E''^\vee (-1)$ on $\Pb(E'')$. By unwinding the definitions (much like in the proof of Proposition \fref{LinearEmbeddingIsVanishingLocus}), one sees that $i^*(s''^\vee)$ corresponds to the identity morphism $E'' \to E''$ on $X$, identifying the corresponding surjection $E''^\vee \to \Oc(1)$ on $\Pb(E'')$ with the canonical one.
\end{proof}

\subsection{Precobordism theories}\label{PrecobordismBackground}

Let us recall from \cite{AY} that a \emph{precobordism theory} $\Bb^*$ is a bivariant theory (on the homotopy category of quasi-projective derived schemes over a Noetherian ring $A$) satisfying certain axioms. What this means is that we get an Abelian group $\Bb^*(X \xrightarrow{f} Y)$ for any morphism $f: X \to Y$, a bilinear \emph{bivariant product} $\bullet$ associated to compositions of morphisms, \emph{bivariant pullbacks} associated to derived Cartesian squares, and \emph{bivariant pushforwards} associated to the factoring of $f$ through a projective morphism $g: X \to X'$. The structure of bivariant theory makes $(\Bb^*(X), \bullet) := (\Bb^*(X \xrightarrow{\mathrm{Id}} X), \bullet)$ (commutative) rings, contravariantly functorial in $X$. The choice of an \emph{orientation} for $\Bb^*$ gives rise to \emph{Gysin pushforward morphisms} $f_!: \Bb^*(X) \to \Bb^*(Y)$
Since we are mostly interested in the cohomology rings in this paper, we are not going to recall bivariant formalism in greater detail here. The interested reader can consult \cite{An} or \cite{AY} for further details.

What we are going to need in this paper, is detailed understanding of the cohomology theory associated to the universal precobordism. It is recorded in the following construction, which is just the cohomological restriction of the bivariant construction of \cite{AY} Section 6.1.

\begin{cons}[Universal precobordism rings]
Let $A$ be a Noetherian ring. For any quasi-projective derived $A$-scheme $X$, we define the \emph{ring of cobordism cycles} $\Mc^*_+(X)$ so that the degree $d$ part $\Mc^d_+(X)$ is the free Abelian group on equivalence classes $[V \xrightarrow{f} X]$ with $f$ projective and quasi-smooth of relative virtual dimension $-d$ (modulo, of course, the relation that disjoint union of cycles corresponds to addition). The (commutative) ring structure is given by the homotopy fibre product over $X$ the class of the identity morphism serving as the unit element. 

We can also define operations on $\Mc_+^*$. For an arbitrary map $g: X \to Y$, the \emph{pullback morphism} $g^*: \Mc_+^*(Y) \to \Mc_+^*(X)$ is defined by linearly extending
$$[V \to Y] \mapsto [V \times_Y X \to X].$$
Note that $g^*$ acts by homomorphisms of rings and preserves degrees. For every projective $g: X \to Y$ that is also quasi-smooth (of relative virtual dimension $-d$), we can define the \emph{Gysin pushforward morphisms} $g_!: \Mc_+^*(X) \to \Mc_+^{*+d}(Y)$ by linearly extending
$$[V \xrightarrow{f} X] \mapsto [V \xrightarrow{g \circ f} Y].$$
Note that $g_!$ doesn't need to preserve the multiplication, only the addition. Both the pullbacks and pushforwards are functorial in the obvious sense.

There are now two equivalent sets of relations one can enforce on the rings $\Mc_+^*$ to obtain the \emph{universal precobordism rings} $\PCob^*$. Either
\begin{enumerate}
\item \emph{double point relations}: given a projective quasi-smooth morphism $W \to \Pb^1 \times X$, let $W_0$ be the fibre over $0$ and suppose that the fibre $W_\infty$ over $\infty$ is the sum of two virtual Cartier divisors $A$ and $B$ inside $W$; we then require that
\begin{equation}\label{DoublePointCobordismEq}
[W_0 \to X] - [A \to X] - [B \to X] + [\Pb_{A \times_W B}(\Oc(A) \oplus \Oc) \to X] = 0
\end{equation}
(the linear combinations of elements of the above form clearly form ideals that are stable in the operations of $\Mc_+^*$, and therefore the quotient theory makes sense);
\end{enumerate}
or
\begin{enumerate}
\item \emph{homotopy fibre relations}: given a projective quasi-smooth morphism $W \to \Pb^1 \times X$, let $W_0$ and $W_\infty$ be the fibres over $0$ and $\infty$ respectively; then we require that
$$[W_0 \to X] - [W_\infty \to X] = 0$$
(again, linear combinations of these form ideals stable in operations of $\Mc_+^*$, making the quotient theory sensible);
\item given line bundles $\Ls_1$ and $\Ls_2$ on $X$, we require that
\begin{align}\label{GeometricFGL}
c_1(\Ls_1 \otimes \Ls_2) &= c_1(\Ls_1) + c_1(\Ls_2) - c_1(\Ls_1) \bullet c_1(\Ls_2) \bullet [\Pb_1 \to X] \\
&- c_1(\Ls_1) \bullet c_1(\Ls_2) \bullet c_1(\Ls_1 \otimes \Ls_2) \bullet ([\Pb_2 \to X] - [\Pb_3 \to X]), \notag 
\end{align}
where 
\begin{align*}
\Pb_1 &:= \Pb_X(\Ls_1 \oplus \Oc); \\
\Pb_2 &:= \Pb_X(\Ls_1 \oplus (\Ls_1 \otimes \Ls_2) \oplus \Oc); \\
\Pb_3 &:= \Pb_{\Pb_X(\Ls_1 \oplus (\Ls_1 \otimes \Ls_2))}(\Oc(-1) \oplus \Oc);
\end{align*}
and $c_1(\Ls)$ is the \emph{first Chern class} (the \emph{Euler class}) of the line bundle $\Ls$ (equivalent to $[V(s) \hookrightarrow X]$ for any global section $s$ of $\Ls$); as this equation is not stable under pushforwards, one needs to take all the generated relations before taking the quotient;
\end{enumerate}
the equivalence of these two sets of relations was proven in \cite{AY} (following the original proof in \cite{LP}). The two slightly different sets of relations will both be convenient for us at certain points of the paper.
\end{cons}

We will also need to recall some of the formal properties of $\PCob^*$ that are helpful in performing computations.

\begin{thm}
The operations of $\PCob^*$ satisfy the following basic properties.
\begin{enumerate}
\item \emph{Push-pull formula:} if  the square
\begin{equation*}
\begin{tikzcd}
X' \arrow[->]{r}{f'} \arrow[->]{d}{g'} & Y' \arrow[->]{d}{g} \\
X \arrow[->]{r}{f} & Y
\end{tikzcd}
\end{equation*}
is homotopy Cartesian with $f$ projective and quasi-smooth (of pure relative virtual dimension), then $f'_! g'^* (\alpha) = g^* f_!(\alpha)$ for all $\alpha \in \PCob^*(X)$.

\item \emph{Projection formula:} if $f: X \to Y$ is projective and quasi-smooth (of pure relative virtual dimension), then $f_!\bigl(f^*(\alpha) \bullet \beta \bigr) = \alpha \bullet f_!(\beta)$ for all $\alpha \in \PCob^*(Y)$ and $\beta \in \PCob^*(X)$. 
\end{enumerate}
\end{thm}
\begin{proof}
Of course, this does not have anything to do with the definition of $\PCob^*$. All the formulas hold already for $\Mc_+^*$ because it is part of a bivariant theory with an orientation (the formulas can also be checked easily on the level of cycles).
\end{proof}

The following simple observations will also be useful for us in the sequel.

\begin{cor}\label{InjPullback}
Let $f: X \to Y$ be projective quasi-smooth morphism between derived schemes so that there exists a class $\eta_X \in \PCob^*(X)$ with $f_! (\eta_X) = 1$. Then the pullback morphism $\pi^*: \PCob^*(Y) \to \PCob^*(X)$ is injective.
\end{cor}
\begin{proof}
Indeed, it follows from the projection formula that
\begin{align*}
f_!(\pi^*(\alpha) \bullet \eta_X) &= \alpha \bullet f_!(\eta_X) \\
&= \alpha
\end{align*}
for all $\alpha \in \PCob^*(Y)$, proving the claim.
\end{proof}

\begin{prop}\label{SummandGivesInjectivePush}
Let $E$ and $F$ be vector bundles on a derived scheme $X$. Then the embedding $i: \Pb(E) \hookrightarrow \Pb(E \oplus F)$ gives rise to an injective Gysin pushforward morphism $i_!: \PCob^*(\Pb(E)) \to \PCob^*(\Pb(E \oplus F)).$ 
\end{prop}
\begin{proof}
By Proposition \fref{BlowUpOfLinearEmbedding} we can form a homotopy Cartesian diagram
\begin{equation*}
\begin{tikzcd}
\Pb(E) \arrow[->]{r}{i'} \arrow[->]{d}{\mathrm{Id}} & \bl_{\Pb(F)}(\Pb(E \oplus F)) \arrow[->]{d}{\pi} \\
\Pb(E) \arrow[->]{r}{i} & \Pb(E \oplus F)
\end{tikzcd}
\end{equation*}
so that $i'$ admits a partial inverse $\rho \circ i' \simeq \mathrm{Id}_{\Pb(E)}$. As $i'_!$ is injective by functoriality, we can use push-pull formula to conclude that also $\pi^* i_!$ is injective, and therefore $i_!$ must be injective.
\end{proof}

\subsubsection*{Precobordism with bundles}

We can also define a slight variant of the above theory called the \emph{universal precobordism ring with vector bundles}. Namely, we start with the group completions $\Mc_+^{d,r}(X)$ of the Abelian monoid on cycles
$$[V \to X, E],$$
where $V \to X$ is quasi-smooth and projective of relative virtual dimension $-d$, $E$ is a vector bundle of rank $r$ on $V$ and the monoid structure is given by disjoint union. One defines the pullbacks and pushforwards for $\Mc_+^{*,*}$ in the obvious way. There are now two product structures that make sense: namely
$$[V_1 \to X, E_1] \bullet_\oplus [V_2 \to X, E_2] := [V_1 \times_X V_2 \to X, E_1 \oplus E_2]$$
and 
$$[V_1 \to X, E_1] \bullet_\otimes [V_2 \to X, E_2] := [V_1 \times_X V_2 \to X, E_1 \otimes E_2].$$
Note that the restricted group $\Mc_+^{*,1}(X)$ with line bundles is a ring with multiplication $\bullet_\otimes$.

The rings $\PCob^{*,*}$ are obtained by enforcing the equation (cf. (\fref{DoublePointCobordismEq}))
$$[W_0 \to X, E \vert_{W_0}] - [A \to X, E \vert_A] - [B \to X, E \vert_B] + [\Pb_{A \times_W B}(\Oc(A) \oplus \Oc) \to X, E \vert_{A \times_W B}] = 0$$
to hold, where $E$ is any vector bundle on $W$. Of course one could also arrive at the same rings by enforcing homotopy fibre relation and (\fref{GeometricFGL}). Note that we have natural embeddings of theories
$$\PCob^*(X) \hookrightarrow \PCob^{*,0}(X)$$
and
$$\PCob^*(X) \hookrightarrow \PCob^{*,1}(X)$$
defined by equipping the cycle with a rank 0 vector bundle or a trivial line bundle respectively. When it should cause no confusion, we often omit the target $X$ from cycles and the subscripts from the products $\bullet_\oplus$ and $\bullet_\otimes$ in order to simplify the notation.

\section{The formal group law of precobordism theory}\label{FGLSect}

The purpose of this section is to prove that the Euler class of $\Ls_1 \otimes \Ls_2$ can be computed easily from those of $\Ls_1$ and $\Ls_2$. Recall first that in \cite{AY} Section 6.3, we were able to find coefficients $a_{i,j} \in \PCob^*(pt)$ of a formal group law
$$F(x,y) = \sum_{i,j} a_{ij} x^i y^j$$
by letting $a_{ij}$ be so that
$$e(\Oc(1,1)) = \sum_{i,j} a_{ij} e(\Oc(1,0))^i \bullet e(\Oc(0,1))^j \in \PCob^*(\Pb^n \times \Pb^m)$$
for any (and hence all) $n \geq i$ and $m \geq j$. It follows from the weak version of the projective bundle formula that this is well defined. Moreover, it is a trivial consequence that given globally generated line bundles $\Ls_1$ and $\Ls_2$ on a quasi-projective derived scheme $X$, then
\begin{equation}\label{FGLEq}
e(\Ls_1 \otimes \Ls_2) = \sum_{i,j} a_{ij} e(\Ls_1)^i \bullet e(\Ls_2)^j \in \PCob^*(X).
\end{equation}
We can now state more precisely the main goal of this section: we want to show that (\fref{FGLEq}) holds for arbitrary $\Ls_1$ and $\Ls_2$, not just globally generated ones (Theorem \fref{FGLTheorem}). Traditionally, there are two ways to reduce the general case to the globally generated one:
\begin{enumerate}
\item using properties of formal group laws, it is possible to define new Euler classes $\tilde e(\Ls)$ in a way that $\tilde e(\Ls) = e(\Ls)$ whenever $\Ls$ is globally generated and (\fref{FGLEq}) holds for arbitrary $\Ls_1, \Ls_2$;
 
\item using Jouanolou's trick one can find a affine scheme $Y$ that is a torsor for a vector bundle on $X$. Then, assuming that the cohomology theory satisfies a strong enough form of homotopy invariance, we can check (\fref{FGLEq}) on by pulling back to $Y$, where the line bundles $\Ls_1$ and $\Ls_2$ become globally generated.

\end{enumerate} 
Since we do not want to redefine Euler classes, and since $\PCob^*$ does not satisfy homotopy invariance, we are forced to do something different. Luckily, everything boils down to explicit computations.

\subsection{Formal group law for arbitrary line bundles}\label{FGLSubSect}

Our strategy is essentially to show that one can derive the formal group law directly from (\fref{GeometricFGL}) without the need of knowing projective bundle formula beforehand. To do this, we need to be able to express the classes $[\Pb_i]$ as power series in $e(\Ls_1)$ and $e(\Ls_2)$ whose coefficients do not depend on $X, \Ls_1$ or $\Ls_2$. From now on, $X$ is assumed to be quasi-projective over a Noetherian ring $A$ of finite Krull dimension and $\Ls_1,\Ls_2$ line bundles on $X$. Let us first recall the following result from \cite{AY}:

\begin{lem}[cf. \cite{AY} Lemma 6.18]\label{ClassOfLine}
We have the equality
\begin{align*}
[X, \Ls] &= \sum_{i \geq 0} e(\Ls)^i \bullet \bigl(\beta_i - [\Pb(\Ls \oplus \Oc)] \bullet \beta_{i-1} \bigr) \in \PCob^{*,1}(X),
\end{align*}
where 
\begin{align*}
(P_0, \Mc_0) &:= \bigl( \Spec(A), \Oc \bigr), \\
P_{i+1} &:= \Pb_{P_i}(\Mc_i \oplus \Oc), \\
\Mc_{i+1} &:= \Mc_{i}(1), \\
\beta_i &:= \pi^*[P_i, \Mc_i] \in \PCob^*(X),
\end{align*}
and $P_{-1} = \emptyset$.
\end{lem}
\begin{proof}
This is what you get when you apply the relation of \cite{AY} Lemma 6.18 infinitely many times to kill all the $T_i$.
\end{proof}

As an easy consequence of the above lemma, we get the following formula taking care of the class $[\Pb_1]$.

\begin{lem}\label{ClassOfProjectivizedLine}
We have the equality
$$[\Pb(\Ls \oplus \Oc)] = { \sum_{i \geq 0} \pi^*[P_{i+1}] \bullet e(\Ls)^i \over \sum_{i \geq 0} \pi^*[P_i] \bullet e(\Ls)^i} \in \PCob^*(X),$$
where $P_i$ are as in Lemma \fref{ClassOfLine}.
\end{lem}
\begin{proof}
It is clear that we have a well defined $\PCob^*$-linear transformation $\Pb: \PCob^{*,1}(X) \to \PCob^{*-1}(X)$ defined by the formula
$$[V \to X, \Ls] \mapsto [\Pb_V(\Ls \oplus \Oc) \to X].$$
Applying this transformation to Lemma \fref{ClassOfLine}, and noting that 
\begin{align*}
\Pb (\beta_i) &= \Pb \pi^* [P_i, \Mc_i] \\
&= \pi^*[\Pb_{P_i}(\Mc_i \oplus \Oc)] \\
&= \pi^*[P_{i+1}],
\end{align*}
we can conclude that
$$[\Pb(\Ls \oplus \Oc)] = \sum_{i \geq 0} e(\Ls)^i \bullet \bigl(\pi^*[P_{i+1}] - [\Pb(\Ls \oplus \Oc)] \bullet \pi^*[P_{i}] \bigr).$$
The claim follows by solving for $[\Pb(\Ls \oplus \Oc)]$.
\end{proof}

It is then easy to find formulas for the classes $[\Pb_2]$ and $[\Pb_3]$

\begin{lem}\label{ClassOfP2P3}
We have the equalities
\begin{align}
[\Pb(\Ls_1 \oplus (\Ls_1 \otimes \Ls_2) \oplus \Oc)] &= \sum_{i,j} \gamma_{ij} e(\Ls_1)^i \bullet e(\Ls_2)^j \\
\notag \\
[\Pb_{\Pb_X(\Ls_1 \oplus (\Ls_1 \otimes \Ls_2))}(\Oc(-1) \oplus \Oc)] &= \sum_{i,j} \phi_{ij} e(\Ls_1)^i \bullet e(\Ls_2)^j
\end{align}
in $\PCob^*(X)$ for some $\gamma_{ij}, \phi_{ij} \in \PCob^*(pt)$ not depending on $X, \Ls_1$ or $\Ls_2$.
\end{lem}
\begin{proof}
The proof is quite easy. Denote by $\Pb$ and $\wtil \Pb$ the $\PCob^*$-linear natural transformations from $\PCob^{i,j}$ to $\PCob^{i-j}$ defined by the formulas
$$[V \to X, E] \mapsto [\Pb_V(E \oplus \Oc) \to X]$$
and
$$[V \to X, E] \mapsto [\Pb_{\Pb_V(E)}(\Oc(-1) \oplus \Oc) \to X]$$
respectively. The desired formula is obtained in each case by applying Lemma \fref{ClassOfLine} to the right hand side of
$$[X, \Ls_1 \oplus (\Ls_1 \otimes \Ls_2)] = [X, \Ls_1] \bullet_{\otimes} \bigl([X, \Oc] \bullet_\oplus [X, \Ls_2] \bigr),$$
replacing the instances of $\PCob(\Ls \oplus \Oc)$ by power series using Lemma \fref{ClassOfProjectivizedLine}, and then finally applying either $\Pb$ or $\wtil \Pb$.
\end{proof}

We are now ready to prove the main theorem of the section.

\begin{thm}\label{FGLTheorem}
Suppose $\pi: X \to \Spec(A)$ is a quasi-projective derived scheme over a Noetherian ring $A$ of finite Krull dimension, and suppose $\Ls_1$ and $\Ls_2$ are line bundles on $X$. Then
$$e(\Ls_1 \otimes \Ls_2) = \sum_{i,j} a_{ij} e(\Ls_1)^i \bullet e(\Ls_2)^j \in \PCob^*(X)$$
where $a_{ij} \in \PCob^*(pt)$ are the coefficients of the formal group law (\fref{FGLEq}) on globally generated line bundles constructed in \cite{AY}.
\end{thm}
\begin{proof}
First of all, we can apply (\fref{GeometricFGL}) to conclude that
$$e(\Ls_1 \otimes \Ls_2) = {e(\Ls_1) + e(\Ls_2) - e(\Ls_1) \bullet e(\Ls_2) \bullet [\Pb_1 \to X] \over 1 + e(\Ls_1) \bullet e(\Ls_2) \bullet ([\Pb_2 \to X] - [\Pb_3 \to X])} \in \PCob^*(X).$$
Lemmas \fref{ClassOfProjectivizedLine} and \fref{ClassOfP2P3} allow us to replace the classes $[\Pb_i \to X]$ by natural power series in $e(\Ls_1)$ and $e(\Ls_2)$, whose coefficients do not depend on $X, \Ls_1$ or $\Ls_2$, and therefore
$$e(\Ls_1 \otimes \Ls_2) = \sum_{i,j} a'_{ij} e(\Ls_1)^i \bullet e(\Ls_2)^j$$
for some $a'_{ij} \in \PCob^*(pt)$ and for all $X, \Ls_1$ and $\Ls_2$. But since the coefficients $a_{ij}$ of (\fref{FGLEq}) were defined by computing the class $e(\Oc(1,1))$ on $\PCob^*(\Pb^n \times \Pb^m)$ as $n$ and $m$ tend to infinity, we see that $a'_{ij}=a_{ij}$ for all $i,j \geq 0$, proving the claim.
\end{proof}

\subsection{Universal property of the universal precobordism}\label{UnivPropSubSect}

The construction of the universal precobordism rings $\PCob^*(X)$ in terms of free generators and relations gives them a universal property. Since the cohomological universal property is not explicitly written down anywhere, we will write down a nice universal property in this subsection (see Theorem \fref{UniversalPropOfPreCob}). In order to state the theorem, we need to make some preliminary definitions. Throughout this section $A$ will be a Noetherian ring of finite Krull dimension, $d \Sc ch_A$ will be the homotopy category of quasi-projective derived schemes over $A$ and $\Rc^*$ is the category of commutative graded rings (not graded commutative). 

\begin{defn}
A functor $\Fc: d \Sc ch_A^\op \to \Rc^*$ is \emph{additive} if the natural inclusions $X \hookrightarrow X \coprod Y$ and $Y \hookrightarrow X \coprod Y$ induce an isomorphism 
$$\Fc \left( X \coprod Y \right) \xrightarrow{\cong} \Fc(X) \times \Fc(Y)$$
of graded rings.
\end{defn}

The main definition, we are going to be working with, is the following: 

\begin{defn}\label{OrientedCohomDef}
An \emph{oriented cohomology theory} $\Hb^*$ on $d \Sc ch_A$ consists of
\begin{enumerate}
\item[($O1$)] an additive functor $\Hb^*: d \Sc ch_A^\op \to \Rc^*$;
\item[($O2$)] for every quasi-smooth projective morphism $f: X \to Y$ of pure relative dimension, we get a \emph{Gysin pushforward morphism}
$$f_!: \Hb^*(X) \to \Hb^*(Y)$$
which is required to be a morphism of $\Hb^*(Y)$-modules, and which is allowed to not preserve the grading.
\end{enumerate}
Note that such a data allows one to define the \emph{Euler class} $e(\Ls)$ of a line bundle $\Ls$ on $X$ as $s_0^* s_{0!} (1_X)$, where $s_0: X \to \Ls$ is the zero section. This data is required to satisfy the following compatibility conditions:
\begin{enumerate}
\item[($Fun$)] the Gysin pushforwards are required to be functorial;

\item[($PP$)] given a homotopy Cartesian diagram
\begin{equation*}
\begin{tikzcd}
X' \arrow[->]{r}{f'} \arrow[->]{d}{g'} & Y' \arrow[->]{d}{g} \\
X \arrow[->]{r}{f} & Y
\end{tikzcd}
\end{equation*}
with $f$ projective and quasi-smooth (of pure relative virtual dimension), then $f'_! g'^* (\alpha) = g^* f_!(\alpha)$ for all $\alpha \in \Hb^*(X)$;

\item[($Norm$)] if $\Ls$ is a line bundle on $X$ and $i: D \hookrightarrow X$ is the inclusion of a derived vanishing locus of a global section of $\Ls$, then 
$$e(\Ls) = i_!(1_D);$$

\item[($FGL$)] the Euler classes of line bundles are nilpotent and there exists elements $b_{ij} \in \Hb^*(pt)$ for $i,j \geq 0$ so that
$$e(\Ls_1 \otimes \Ls_2) = \sum_{i,j} b_{ij} e(\Ls_1)^i \bullet e(\Ls_2)^j$$
for all $X, \Ls_1$ and $\Ls_2$.
\end{enumerate}
\end{defn}

\begin{ex}
Both $K^0$ and $\PCob^*$ are oriented cohomology theories on $d \Sc ch_A$. Note that we have to regard $K^0(X)$ as a graded ring concentrated in degree 0, and therefore we do not want to make restrictions on how Gysin pushforwards affect the grading. 
\end{ex}

\begin{rem}
Why does this definition deserve the name ``oriented cohomology theory''? It does not look at all similar to the definition used by Levine and Morel in \cite{LM}. For example, here the formal group law is one of the axioms whereas in \cite{LM} it was a consequence of the projective bundle formula (which is not an axiom for us). But as we already noted in Section \fref{FGLSubSect}, the failure of homotopy property seems to make it impossible to deduce the validity of formal group laws for line bundles that are not globally generated. Hence, Levine-Morel style characterization seems not to produce a useful notion in this generality. 
\end{rem}

Of course, it will turn out that $\PCob^*$ is the universal oriented cohomology theory. Let us start with an easy Lemma.

\begin{lem}\label{MapFromCycleGroup}
Suppose $\Hb^*$ is a oriented cohomology theory on $d \Sc ch_A$. Then there is a unique natural transformation $\eta': \Mc^*_+ \to \Hb^*$ commuting with pullbacks, pushforwards and preserving the ring structure.
\end{lem}
\begin{proof}
Indeed, such a transformation must preserve fundamental classes, and therefore
\begin{align*}
\eta \bigl([V \xrightarrow{f} X] \bigr) &= \eta' \bigl( f_!(1_V) \bigr) \\
&= f_! \bigl( \eta' (1_V) \bigr) \\
&= f_! \bigl( 1_V^{\Hb^*} \bigr),
\end{align*}
where $1_V^{\Hb^*} \in \Hb^*(V)$ is the multiplicative unit. The morphisms $\eta'$ preserve addition since $\Hb^*$ was assumed to be additive. They commute with pushforwards by $(Fun)$ and with pullbacks by $(PP)$. To prove that $\eta'$ preserves multiplication, we compute
\begin{align*}
\eta'([V_1 \xrightarrow{f_1} X]) \bullet \eta'([V_2 \xrightarrow{f_2} X]) &= f_{1!}(1_{V_1}^{\Hb^*}) \bullet \eta'([V_2 \xrightarrow{f_2} X]) \\
&= f_{1!} \bigl(f_1^*(\eta'([V_2 \xrightarrow{f_2} X]))  \bigr) \quad (\text{linearity of $f_!$}) \\
&= \eta'\bigl( f_{1!} (f_1^*([V_2 \xrightarrow{f_2} X]))  \bigr) \\
&= \eta' \bigl( [V_1 \xrightarrow{f_1} X] \bullet [V_2 \xrightarrow{f_2} X] \bigr). 
\end{align*}
We have shown that $\eta'$ satisfies all the desired properties, so we are done.
\end{proof}

In order to show that the above morphisms $\eta$ descend the double point cobordism relations, we need the following result (cf. \cite{LP} Lemma 3.3).

\begin{lem}\label{ClassOfProjectivizedLine2}
Let $\Hb^*$ be an oriented cohomology theory, and let $X \in d\Sc ch_A$. If $\Ls$ is a line bundle on $X$, then 
$$\pi_! \bigl( 1_{\Pb(\Ls \oplus \Oc)} \bigr) = -\sum_{i,j \geq 1} b_{ij} e(\Ls)^{i-1} \bullet e(\Ls^\vee)^{j-1}$$
where $b_{ij}$ are as in $(FGL)$ and $\pi$ is the natural projection $\Pb(\Ls \oplus \Oc) \to X$. 
\end{lem}
\begin{proof}
The proof is the same as in \cite{LP} with the distinction that all blow ups should be derived.
\end{proof}

We are finally ready to prove the main result of this subsection: $\PCob^*$ is the universal oriented cohomology theory.

\begin{thm}\label{UniversalPropOfPreCob}
Suppose $\Hb^*$ is a oriented cohomology theory on $d \Sc ch_A$. Then there is a unique natural transformation $\eta: \PCob^* \to \Hb^*$ commuting with pullbacks, pushforwards and preserving the ring structure.
\end{thm}
\begin{proof}
We only have to show that the natural transformation $\eta'$ of \label{MapFromCycleGroup} descends the double point cobordism relations (\fref{DoublePointCobordismEq}). What this boils down to is showing that 
$$\eta' \bigl([\Pb_{A \times_W B}(\Oc(A) \oplus \Oc) \to X] \bigr) = - \sum_{i,j \geq 1} b_{ij} e(\Oc(A))^i \bullet e(\Oc(B))^j \in \Hb^*(W).$$
Denote by $i$ the embedding $A \times_W B \hookrightarrow W$. We note that $i_!(1_{A \times_W B})$ is just the product $e(\Oc(A)) \bullet e(\Oc(B))$, and the restrictions of $\Oc(A)$ and $\Oc(B)$ to $A \times_W B$ are duals of each other. We can now compute that
\begin{align*}
\eta' \bigl([\Pb_{A \times_W B}(\Oc(A) \oplus \Oc) \to X] \bigr) &= - i_! \bigl( \sum_{i,j \geq 1} b_{ij} e(\Oc(A))^{i-1} \bullet e(\Oc(-A))^{j-1} \bigr) \quad (\text{Lemma \fref{ClassOfProjectivizedLine2}}) \\
&=- i_! \bigl( \sum_{i,j \geq 1} b_{ij} e(\Oc(A))^{i-1} \bullet e(\Oc(B))^{j-1} \bigr) \\
&= \sum_{i,j \geq 1} b_{ij} e(\Oc(A))^{i} \bullet e(\Oc(B))^{j}, \quad (\text{linearity of $i_!$})
\end{align*}
which finishes the proof.
\end{proof}

\section{Chern classes in precobordism rings}\label{ChernClassConsSect}

The purpose of this section is to construct \emph{Chern classes} for the precobordism rings $\PCob^*(X)$. More precisely, given a quasi-projective derived scheme over a Noetherian ring $A$ and a vector bundle $E$ on $X$ of rank $r$, we would like to construct classes
$$c_i(E) \in \PCob^i(X)$$
for $1 \leq i \leq r$. We will also show that the newly constructed Chern classes satisfy the usual expected properties.

Before beginning the construction, we recall that given a vector bundle $E$ of rank $r$ on a derived scheme $X$, we can define its \emph{Euler class} (or \emph{top Chern class}) $e(E) \in \PCob^r(X)$ as the cycle class $[V(s) \to X]$ of the inclusion of the derived vanishing locus of any global section $s$ of $E$ (the class does not depend on $s$). The following easy results about Euler classes will be used in the construction of Chern classes.

\begin{lem}\label{MultTopChern}
Let $X$ be a quasi-projective derived scheme over a Noetherian ring $A$, and let 
$$0 \to E' \to E \to E'' \to 0$$
be a short exact sequence of vector bundles of rank $r', r, r''$ respectively on $X$. Then 
$$e(E) = e(E') \bullet e(E'')$$
in $\PCob^r(X)$.
\end{lem}
\begin{proof}
Let $s$ be a global section of $E$, which maps to a global section $s''$ of $E''$. Denote the inclusion $V(s'') \hookrightarrow X$ by $i$. As $e(E'') = i_!(1_{V(s'')})$, we can use the projection formula to conclude that 
\begin{align*}
e(E') \bullet e(E'') &= e(E') \bullet i_!(1_{V(s'')}) \\
&= i_! \bigl(e(i^* E') \bigr).
\end{align*} 
We can now use the canonical section $s'$ of $i^* E'$ as in Construction \fref{VanishingLociInExtensions} to see that $e(i^*E')$ is represented by $[V(s) \to V(s'')]$, where $V(s)$ is the vanishing locus of $s$ in $X$. Hence $i_!\bigl(e(i^*E')\bigr) = [V(s) \to X]$ and the claim follows.
\end{proof}

\begin{lem}\label{NilpTopChern}
Let $X$ be a quasi-projective derived scheme over a Noetherian ring $A$, and let $E$ be a vector bundle of rank $r$ on $X$. Then the Euler class $e(E)$ of $E$ is a nilpotent element of $\PCob^*(X)$.
\end{lem}
\begin{proof}
The case $r=1$ is Lemma 6.2 of \cite{AY}. Note that if a section $s$ of $E^\vee \otimes \Ls$ doesn't vanish on some open subset $U \subset X$, then the naturally associated map $s' \in \Hom_X(E, \Ls)$ is surjective on $U$. As $X$ is quasi-projective and $A$ is Noetherian, we can find a line bundle $\Ls$ and global sections $s_1,...,s_n \in \Gamma(X; E^\vee \otimes \Ls)$ so that the total vanishing locus of the sections $s_i$ is empty. It follows that they induce a surjective morphism $E^{\oplus n} \to \Ls$ of vector bundles, and therefore we must have
$$e(E)^n = e(E^{\oplus n}) = \alpha \bullet e(\Ls)$$
for some $\alpha \in \PCob^{rn-1}(X)$ by Lemma \fref{MultTopChern}. The claim now follows from the nilpotence of $e(\Ls)$ and the commutativity of the product $\bullet$.
\end{proof}
\begin{rem}
There are much more natural proofs of Lemma \fref{NilpTopChern}, for example by following ideas of \cite{AY} Section 6.2 and reducing to the globally generated case by deforming the vector bundle $E \otimes \Ls$ to $E$. Such an approach would have the advantages of being much more natural and yielding a precise formula for $e(E)$ in terms of $e(E)$, $e(\Ls)$ and some geometric data, but it would have the disadvantage of being considerably longer.
\end{rem}

\subsection{Construction of Chern classes}\label{ChernClassConsSubSect}

In this section, we are going to construct the Chern classes. The key fact that we are going to need is the following observation:

\begin{lem}\label{FiltrationOnBlowUpLem}
Let $X$ be a derived scheme, $E$ a vector bundle on $X$ and $s \in \Gamma(X;E)$ a global section of $E$. Now the canonical equivalence of \fref{ProjectiveBlowUp} and the canonical surjection $E^\vee \to \Oc(1)$ on $\Pb(E)$ yield us a natural surjection
\begin{equation}\label{FundamentalSurjection}
E^\vee \to \Oc(-\Ec)
\end{equation}
of vector bundles on $\bl_Z(X)$, where $\Ec$ is the exceptional divisor.
\end{lem}
\begin{proof}
The only thing that hasn't been explicitly stated already is that the line bundle $\Oc(1)$ on $\Pb(E)$ restricts to $\Oc(-\Ec)$ on the blow up $\bl_Z(X)$. But this is easy as $\Ec$ is by construction the vanishing locus of a global section of restriction of $\Oc(-1)$.
\end{proof}

Another result we are going to need in the construction is the following.

\begin{lem}\label{BreakingApartUnity}
Let $X$ be a quasi-projective derived scheme over a Noetherian ring $A$ of finite Krull dimension, and let $E$ be a vector bundle on $X$. Let $s \in \Gamma(X; E)$ be a global section of $E$, and let $Z \hookrightarrow X$ be the inclusion of the derived vanishing locus of $s$. Denote by $\wtil X$ the derived blow up $\bl_Z(X)$. Then
$$1_X = \sum_{i=0}^\infty e(E)^i \bullet [\Pb(E \oplus \Oc) \to X]^i \bullet \bigl([\wtil X \to X] - [\Pb_\Ec(\Oc(\Ec) \oplus \Oc) \to X]\bigr) \in \PCob^*(X),$$
where $\Ec$ is the exceptional divisor of $\wtil X$.
\end{lem}
\begin{proof}
Consider $W := \bl_{\infty \times Z} (\Pb^1 \times X) \to \Pb^1 \times X$ and note that the fibre of $W$ over $\infty$ is the sum of virtual Cartier divisors $\bl_Z(X)$ and the exceptional divisor $\Ec' \simeq \Pb_Z(E \oplus \Oc)$ intersecting at the exceptional divisor $\Ec$ of $\bl_Z(X)$. The morphism $W \to \Pb^1 \times X$ is therefore a derived double point degeneration over $X$ realizing the relation
\begin{align*}
1_X &= [\wtil X \to X] + [\Pb_Z(E \oplus \Oc) \to X] - [\Pb_\Ec(\Oc(\Ec) \oplus \Oc) \to X] \\
&= [\wtil X \to X] + e(E) \bullet [\Pb(E \oplus \Oc) \to X] - [\Pb_\Ec(\Oc(\Ec) \oplus \Oc) \to X]
\end{align*}
in $\PCob^*(X)$. Moreover, $\Pb_W(E \oplus \Oc) \to \Pb^1 \times X$ realizes the relation
\begin{align*}
[\Pb(E \oplus \Oc) \to X] &= [\Pb_{\wtil X}(E \oplus \Oc) \to X] + e(E) \bullet [\Pb(E \oplus \Oc) \to X]^2 \\
&- [\Pb(E \oplus \Oc) \to X] \bullet [\Pb_\Ec(\Oc(\Ec) \oplus \Oc) \to X]
\end{align*}
and more generally, the $n$-fold derived fibre product of $\Pb_W(E \oplus \Oc)$ over $W$ realizes the relation
\begin{align*}
[\Pb(E \oplus \Oc) \to X]^n &= [\Pb(E \oplus \Oc) \to X]^n \bullet [\wtil X \to X] + e(E) \bullet [\Pb(E \oplus \Oc) \to X]^{n+1} \\
&- [\Pb(E \oplus \Oc) \to X]^n \bullet [\Pb_\Ec(\Oc(\Ec) \oplus \Oc) \to X].
\end{align*}
Combining these these relations and remembering that the Euler classes are nilpotent, we obtain the desired equation. 
\end{proof}

We can now begin our construction.

\begin{cons}[Construction of $c_i(E)$]\label{ChernClassCons}
Suppose $X$ is a quasi-projective derived scheme over a Noetherian ring $A$ of finite Krull dimension, and let $E$ be a vector bundle of rank $r$ on $X$. We are going to construct 

\begin{enumerate}
\item a projective quasi-smooth morphism $\pi_{X,E}: \wtil X_E \to X$ of relative virtual dimension 0 that is moreover natural in pullbacks in the obvious sense;

\item $\pi_{X,E}^* E$ has a natural filtration $E_\bullet$ by vector bundles with line bundles $\Ls_1, ... , \Ls_r$ as the associated graded pieces;

\item a class $\eta_{X,E} \in \PCob^0(\wtil X_E)$ pushing forward to $1_X \in \PCob^0(X)$ that is natural in the sense that given $f: Y \to X$, then if $f'$ is as in the homotopy Cartesian square (by the first item)
\begin{center}
\begin{tikzcd}
\wtil Y_{f^*E} \arrow[->]{r}{f'} \arrow[->]{d}{\pi_{Y,f^*E}} & \wtil X_E \arrow[->]{d}{\pi_{X,E}} \\
Y \arrow[->]{r}{f} & X,
\end{tikzcd}
\end{center}
we have $f'^*(\eta_{X,E}) = \eta_{Y,f^*E}$.
\end{enumerate}
After finding such data, we can define the \emph{Chern classes}
\begin{equation}\label{ChernClassDef}
c_i(E) := \pi_{X,E!} \bigl( s_i(e(\Ls_1),...,e(\Ls_r)) \bullet \eta_{X,E} \bigr)
\end{equation}
where $s_i$ is the $i^{th}$ elementary symmetric polynomial, and the \emph{total Chern class}
\begin{equation}\label{TotalChernClassDef}
c(E) := 1 + c_1(E) + \cdots + c_r(E).
\end{equation}

In order to perform the desired construction, we are going to proceed by induction on the rank $r$. For the base case $r=1$, we set $\pi_{X,\Ls}$ to be the identity morphism $X \to X$, the filtration to be the trivial filtration $0 \subset \Ls$, and the class $\eta_{X, \Ls} \in \PCob^0(X)$ to be $1_X$. 

Suppose then that $r > 1$ and that we have performed the desired construction for all derived schemes $Y$ and all vector bundles $F$ of rank at most $r-1$, and let $E$ be a rank $r$ vector bundle on a derived scheme $X$. Let $Z \hookrightarrow X$ be the derived vanishing locus of the zero section of $E$, and denote by $\wtil X$ the derived blow up $\bl_Z(X)$ whose structure morphism we are going to denote by $\pi$. Note that by Lemma \fref{FiltrationOnBlowUpLem}, there is a natural short exact sequence
\begin{equation}\label{QuotientSeqOnBlowUp}
0 \to \Oc(\Ec) \to E \to Q \to 0
\end{equation}
of vector bundles on $\wtil X$, which provides the first step of the desired filtration.  Moreover, by Lemma \fref{BreakingApartUnity}, the class
$$\eta'_{X,E} := \sum_{i=0}^\infty e(E)^i \bullet [\Pb_{\wtil X}(E \oplus \Oc) \to \wtil X]^i \bullet \bigl(1_{\wtil X} - e(\Oc(\Ec)) \bullet [\Pb_{\wtil X} (\Oc(\Ec) \oplus \Oc) \to \wtil X] \bigr) \in \PCob^*(\wtil X)$$
pushes forward to $1_X \in \PCob^*(X)$. We can now apply the inductive argument to the pair $\wtil X = \bl_Z(X)$, $Q$ to obtain a natural map $\pi_{\bl_Z(X), Q}: \wtil{\bl_Z(X)}_Q \to \bl_Z(X)$, a natural filtration of the vector bundle $\pi_{\bl_Z(X), Q}^*Q$, and a natural class $\eta_{\bl_Z(X),Q} \in \PCob^0(\wtil{\bl_Z(X)}_Q)$ having the desired properties. We now set
\begin{equation}\label{FundamentalMapDef}
\pi_{X,E}: \wtil X_E := \wtil{\bl_Z(X)}_Q \xrightarrow{\pi_{\bl_Z(X), Q}} \bl_Z(X) \xrightarrow{\pi} X
\end{equation}
and
\begin{equation}\label{FundamentalClassDef}
\eta_{X,E} := \eta_{\bl_Z(X), Q} \bullet \pi_{\bl_Z(X), Q}^*(\eta'_{X,E}).
\end{equation}
Moreover, we obtain a natural filtration $E_\bullet$ of $E$ on $\wtil X_E$ by setting $E_1$ to be the pullback of the canonical inclusion $\Oc(\Ec) \to E$ to $\wtil X_E$, and then pulling back the natural filtration of $Q$ along the surjection $E \to Q$. To check that $\eta_{X,E}$ pushes forward to $1_X$, we compute that
\begin{align*}
\pi_{X,E!}(\eta_{X,E}) &= \pi_! \Bigl( \pi_{\bl_Z(X), Q!} \bigl( \eta_{\bl_Z(X), Q} \bullet \pi_{\bl_Z(X), Q}^*(\eta'_{X,E}) \bigr) \Bigr) \\
&= \pi_! \Bigl( \pi_{\bl_Z(X), Q!} \bigl( \eta_{\bl_Z(X), Q} \bigr) \bullet \eta'_{X,E} \Bigr) \quad (\text{projection formula}) \\
&= \pi_! \Bigl( 1_{\bl_Z(X)} \bullet \eta'_{X,E} \Bigr) \quad (\text{induction}) \\
&= \pi_! \Bigl(\eta'_{X,E} \Bigr) \\
&= 1_X. 
\end{align*}
We can therefore define Chern classes using formula (\fref{ChernClassDef}). 
\end{cons}

\subsection{Basic properties of Chern classes}\label{ChernPropertiesSubSect}

Let us start by showing that Chern classes have many desirable properties.

\begin{thm}\label{ChernProperties}
Let $X$ be a quasi-projective derived $A$-scheme for a Noetherian ring $A$, and let $E$ be a vector bundle of rank $r$ on $X$. Now the Chern classes of Construction \fref{ChernClassCons} satisfy the following basic properties:
\begin{enumerate}
\item \emph{Naturality:} if $f: Y \to X$ is a quasi-projective map of derived schemes, then $f^*c_i(E) = c_i(f^* E) \in \PCob^r(Y)$.

\item \emph{Normalization:} $c_r(E) = e(E) \in \PCob^r(X)$.
\end{enumerate}
\end{thm}
\begin{proof}
\noindent
\begin{enumerate}
\item This follows immediately from Construction \fref{ChernClassCons}, since it is completely natural in pullbacks.

\item By the definition (\fref{ChernClassDef}) of the top Chern class, we have
\begin{align*}
c_r(E) &= \pi_{X,E!} \bigl( e(\Ls_1)\bullet \cdots \bullet  e(\Ls_r) \bullet \eta_{X,E} \bigr) \\
&= \pi_{X,E!} \bigl( e(E) \bullet \eta_{X,E} \bigr) \quad (\text{Lemma \fref{MultTopChern}}) \\
&= e(E) \bullet \pi_{X,E!} \bigl( \eta_{X,E} \bigr) \quad (\text{projection formula}) \\
&= e(E),
\end{align*}
proving the claim. \qedhere
\end{enumerate}
\end{proof}

Next we are going to show that the classical formula satisfied by the Chern classes of $\Oc(1)$ in the cohomology ring of $\Pb(E)$ holds in our precobordism rings $\PCob^*(\Pb(E))$ as well. This result will be used later in the proof of the Projective Bundle Formula (Theorem \fref{PBF}). The proof follows closely to the proofs of Lemmas  4.1.18 and 4.1.19 in \cite{LM}.

\begin{thm}\label{ProjectiveBundleChernClass}
Let $X$ be a quasi-projective derived scheme over a Noetherian ring $A$ of finite Krull dimension and let $E$ be a vector bundle of rank $r$ on $X$. Then
\begin{align}\label{PbEq1}
0 &= e(E(1)) = \sum_{i = 0}^r (-1)^{i} c_{r-i}(E) \bullet c_1(\Oc_{\Pb(E)} (-1))^i 
\end{align}
and
\begin{align}\label{PbEq2}
0 = e(E^\vee (-1)) = \sum_{i = 0}^r (-1)^{i} c_{r-i}(E^\vee) \bullet c_1(\Oc_{\Pb(E)} (1))^i = 0
\end{align}
in $\PCob^*(\Pb(E))$.
\end{thm}
\begin{proof}
We start by considering the vector bundle $E(1)$ on $\Pb(E)$. By basic properties of projective bundles, we have a natural identification of spaces
$$\Gamma\bigl(\Pb(E); E(1)\bigr) \simeq \Hom_X(E, E);$$
let $s$ be the global section of $E(1)$ corresponding to the identity morphism $E \to E$. It is easy to verify locally that the derived vanishing locus of $s$ is empty so that $c_r(E(1)) = 0$. The triviality of $e(E^\vee (-1))$ follows then from Lemma \fref{TrivTopChernOfDual} and the formulas (\fref{PbEq1}) and (\fref{PbEq2}) follow from Lemma \fref{TrivTopChernOfTwist}.
\end{proof}

\begin{lem}\label{TrivTopChernOfDual}
Let $E$ be a vector bundle on a quasi-projective derived scheme $X$ over a Noetherian ring $A$ of finite Krull dimension. If $e(E) = 0$, then $e(E^\vee) = 0$.
\end{lem}
\begin{proof}
By naturality, $e(E^\vee) = c_r(E)$. Consider the map $\pi_{X,E^\vee}: \wtil X_{E^\vee} \to X$, and recall that $E^\vee$ has a natural filtration on $\wtil X_{E^\vee}$ with line bundles $\Ls_i$ as associated graded pieces. By dualizing, we obtain a natural filtration of $E$ \emph{on $\wtil X_{E^\vee}$}, with associated graded pieces $\Ls_i^\vee$. Since $\PCob^*$ has a formal group law, we have 
\begin{align}\label{ChernClassOfDual}
e(\Ls_i) &= e(\Ls_i^\vee) \bullet \sum_j a_j e(\Ls_i^\vee)^j \\
&=: e(\Ls_i^\vee) \bullet \rho (\Ls_i^\vee) \notag
\end{align}  
for some $a_j \in \PCob^*(\Spec(A))$ (and even $a_0 = -1$).

We can then compute that
\begin{align*}
c_r(E^\vee) &= \pi_{X,E^\vee!} \bigl( e(\Ls_1) \bullet \cdots \bullet e(\Ls_r) \bullet \eta_{X,E^\vee} \bigr) \\
&=  \pi_{X,E^\vee!} \bigl( e(\Ls_1^\vee) \bullet \rho (\Ls_1^\vee) \bullet \cdots \bullet e(\Ls_r^\vee) \bullet \rho(\Ls_r^\vee) \bullet \eta_{X,E^\vee} \bigr) \quad (\fref{ChernClassOfDual}) \\
&= \pi_{X,E^\vee!} \bigl( e(E) \bullet \rho (\Ls_1^\vee) \bullet \cdots \bullet \rho(\Ls_r^\vee) \bullet \eta_{X,E^\vee} \bigr) \quad (\text{Lemma \fref{MultTopChern}}) \\
&= \pi_{X,E^\vee!} \bigl( 0 \bullet \rho (\Ls_1^\vee) \bullet \cdots \bullet \rho(\Ls_r^\vee) \bullet \eta_{X,E^\vee} \bigr) \\
&= 0,
\end{align*}
which proves the claim.
\end{proof}

\begin{lem}[Cf. \cite{LM} Lemma 4.1.18]\label{TrivTopChernOfTwist}
Let $E$ be a vector bundle on a quasi-projective derived scheme $X$ over a Noetherian ring $A$ of finite Krull dimension. If $e(E \otimes \Ls) = 0$, then
$$\sum_{i=0}^r (-1)^i c_{r-i}(E) \bullet c_1(\Ls^\vee)^i = 0.$$
\end{lem}
\begin{proof}
Let $\Ls_i$ be the line bundles associated to the natural filtration of $E$ on $\wtil X_{E}$, and let $F$ denote the formal group law of the theory $\PCob^*$. We now have that
$$\Bigl(F\bigl(e(\Ls_1 \otimes \Ls), e(\Ls^\vee)\bigr) - e(\Ls^\vee) \Bigr) \bullet \cdots \bullet \Bigl(F\bigl(e(\Ls_r \otimes \Ls), e(\Ls^\vee)\bigr) - e(\Ls^\vee) \Bigr) = 0$$
since the product is clearly divisible by $e(\Ls_1 \otimes \Ls) \bullet \cdots \bullet e(\Ls_r \otimes \Ls) = e(E) = 0$. On the other hand, the above implies that 
$$\bigl(e(\Ls_1) - e(\Ls^\vee) \bigr) \bullet \cdots \bullet \bigl( e(\Ls_r) - e(\Ls^\vee) \bigr) = 0 \in \PCob(\wtil X_E),$$
and the desired formula follows from the definition (\fref{ChernClassDef}) of Chern classes after pushing forward along $\pi_{X,E}$ and applying projection formula.
\end{proof}

\subsection{Splitting principle and further properties of Chern classes}\label{SplittingPrincipleSubSect}

The purpose of this section is to derive further desirable properties of Chern classes that are going to be necessary in Section \fref{ApplicationsSect}. The main tool is going to be the following theorem:

\begin{thm}[Splitting principle]\label{SplittingPrinciple}
Suppose $X$ is a derived scheme over a Noetherian ring $A$ of finite Krull dimension, and suppose $E$ is a vector bundle of rank $r$ on $X$. If $E$ has a filtration with line bundles $\Ls_i$ as the associated graded pieces ($E$ \emph{has a splitting} by $\Ls_i$), then
$$c_i(E) = s_i(e(\Ls_1),...,e(\Ls_r)),$$
where $s_i$ is the $i^{th}$ elementary symmetric polynomial.
\end{thm}

The above theorem will follow easily from Theorem \fref{ProjectiveBundleChernClass} after we prove the following lemma.

\begin{lem}\label{PBFInj}
Let $X$ be a quasi-projective derived scheme over a Noetherian ring $A$ of finite Krull dimension, and let $E$ be a vector bundle of rank $r$ on $X$. Then the morphism
$$\Ps roj: \bigoplus_{i=0}^{r-1} \PCob^{*-i}(X) \to \PCob^*(\Pb(E))$$
defined by
$$(\alpha_0, \alpha_1, ..., \alpha_{r-1}) \mapsto \sum_{i=0}^{r-1} \alpha_i \bullet e(\Oc(1))^i$$
is an injection.
\end{lem} 

\begin{rem}
Of course, later we will show that the above map is surjective as well, giving rise to the projective bundle formula (Theorem \fref{PBF}). Unfortunately, the injectivity part seems to be necessary for the proof of splitting principle, and that in turn seems to be necessary to conclude that all Chern classes are nilpotent. Nilpotence of Chern classes is used in the proof Theorem \fref{PBF}.
\end{rem}

\begin{proof}
Note that we can assume that $E$ has a splitting by some line bundles $\Ls_1,...,\Ls_r$ on $X$: by Corollary \fref{InjPullback} the pullback morphism $\PCob^*(X) \to \PCob(\wtil X_E)$ is an injection (see Construction \fref{ChernClassCons}) and moreover the pullback of $E$ splits on $\wtil X_E$.

Since $X$ is quasi-projective we can find a short exact sequence
$$0 \to \Ls \otimes E \to \Oc_X^{\oplus n+1} \to F \to 0$$ 
of vector bundles on $X$ inducing a linear embedding $i: \Pb(E) \hookrightarrow \Pb_X^n$ with the property that 
\begin{equation}\label{PullbackOfAntitautulogical}
i^* \Ls(1) \simeq \Oc(1).
\end{equation}
Of course, the advantage of this embedding is that the structure of $\PCob^*(\Pb^n_X)$ is understood: by the results of \cite{AY} $\PCob^*(\Pb^n_X) \cong \PCob^*(X)[t]/(t^{n+1})$, where $t = e(\Oc(1))$.

By Proposition \fref{LinearEmbeddingIsVanishingLocus}, $\Pb(E)$ is the derived vanishing locus of a section of $F(1)$. Without loss of generality we may assume that $F$ has a splitting by line bundles $\Ms_1, ..., \Ms_s$ on $X$ (note that $r+s=n$) so that 
\begin{align*}
i_!(1_{\Pb(E)}) &= e(F(1)) \\
&= F\bigl(e(\Ms_1), e(\Oc(1)) \bigr) \bullet \cdots \bullet F\bigl(e(\Ms_r), e(\Oc(1)) \bigr) \quad (\text{Proposition \fref{MultTopChern}}), 
\end{align*}
where $F$ is the formal group law of $\PCob^*$. Expanding the right hand side of the equation, we see that
\begin{align}\label{PresentationOfUnit}
i_!(1_{\Pb(E)}) &= \sum_{j=0}^n \beta_j \bullet e(\Oc(1))^j
\end{align}
where $\beta_j \in \PCob^*(X)$ is nilpotent for $j \not = n-s = r$ and a unit for $j = n-s = r$. Using this, we can conclude that for $l \leq r$
\begin{align*}
i_!\bigl( e(\Oc(1))^l \bigr) &=e(\Ls(1)) \bullet i_!(1_{\Pb(E)}) \quad (\text{projection formula and (\fref{PullbackOfAntitautulogical})}) \\
&= F \bigl(e(\Ls), e(\Oc(1)) \bigr) \bullet e(F(1)) \\
&= \sum_{j=0}^n \beta_{l,j} \bullet e(\Oc(1))^j \quad (\text{\fref{PresentationOfUnit}})
\end{align*}
where $\beta_{l,j} \in \PCob^*(X)$ is nilpotent for $j \not = r - l$ and a unit for $j = r - l$.

To put the above in other words, the images of the basis vectors in 
$$i_! \circ \Ps roj: \bigoplus_{i=0}^{r-1} \PCob^{*-i}(X) \to \PCob^*(\Pb^n_X)$$ 
are linearly independent over $\PCob^*(X)$ and therefore $i_! \circ \Ps roj$ is injective. But this implies the injectivity of $\Ps roj$, proving the claim.
\end{proof}

We can now prove the validity of splitting principle:

\begin{proof}[Proof of Theorem \fref{SplittingPrinciple}]
By Theorem \fref{ProjectiveBundleChernClass}, we have an equality
\begin{equation*}
\sum_{i=0}^r (-1)^i c_{r-i}(E) \bullet e(\Oc(1))^i = 0 \in \PCob^*(\Pb(E^\vee)).
\end{equation*}
On the other hand, since $E$ has a splitting by line bundles $\Ls_i$, we can argue as in the proof of Lemma \fref{TrivTopChernOfTwist} that
\begin{equation*}
\sum_{i=0}^r (-1)^i s_{r-i}\bigl(e(\Ls_1), ..., e(\Ls_r)\bigr) \bullet e(\Oc(1))^i = 0 \in \PCob^*(\Pb(E^\vee)).
\end{equation*}
The left hand sides of both formulas both share the term $(-1)^r e(\Oc(1))^r$, so we can conclude that
\begin{equation*}
\sum_{i=0}^{r-1} (-1)^i c_{r-i}(E) \bullet e(\Oc(1))^i = \sum_{i=0}^{r-1} (-1)^i s_{r-i}\bigl(e(\Ls_1), ..., e(\Ls_r)\bigr) \bullet e(\Oc(1))^i.
\end{equation*}
It then follows from Lemma \fref{PBFInj} that
$$c_i(E) = s_i(e(\Ls_1),...,e(\Ls_r)),$$
which is exactly what we wanted.
\end{proof}

As an immediate application we prove the following properties of Chern classes, which will be useful later.

\begin{thm}[Nilpotence of Chern classes]\label{NilpChernClass}
Let $X$ be a quasi-projective derived scheme over a Noetherian ring $A$ of finite Krull dimension and suppose $E$ is a vector bundle on $X$. Then the Chern classes $c_i(E) \in \PCob^*(X)$ are nilpotent.
\end{thm}
\begin{proof}
Since Euler classes are nilpotent, the theorem follows trivially from Theorem \fref{SplittingPrinciple} whenever $E$ splits. On the other hand, the pullback morphism
$$\pi_{X,E}^*: \PCob^*(X) \to \PCob^*(\wtil X_E)$$
(see Construction \fref{ChernClassCons}) is injective by Corollary \fref{InjPullback}, and as $E$ splits on $\wtil X_E$, the claim follows.
\end{proof}

\begin{thm}[Whitney sum formula]\label{WSF}
Let $X$ be a quasi-projective derived scheme over a Noetherian ring $A$ of finite Krull dimension and let 
$$0 \to E' \to E \to E'' \to 0$$
be a short exact sequence of vector bundles on $X$. Then 
$$c(E) = c(E') \bullet c(E'') \in \PCob^*(X).$$
\end{thm}
\begin{proof}
Again, the theorem follows trivially from Theorem \fref{SplittingPrinciple} whenever both $E'$ and $E''$ split. There are multiple ways to find a map $\pi: \wtil X \to X$ so that $\pi^*$ is injective and the desired splitting occurs on $\wtil X$, so the claim follows.
\end{proof}

\section{Applications}\label{ApplicationsSect}

The purpose of this section is to use the theory of Chern classes to study the relationship between the universal precobordism theory $\PCob^*(X)$, $K$-theory and intersection theory, as well as strengthen the weak projective bundle theorem from \cite{AY}. In Section \fref{CFSubSect} we will show that the algebraic $K$-theory ring $K^0(X)$ can be recovered in a simple way from the universal precobordism ring $\PCob^*(X)$. In Section \fref{RRSubSect}, we will study a candidate theory for the Chow-cohomology of $X$ (the \emph{universal additive precobordism}, see Definition \fref{AdditivePCob} below), and show that the classical Grothendieck-Riemann-Roch theorem generalizes to this setting.

Before beginning, we need some preliminary definitions. Recall (from Section \fref{PrecobordismBackground}) that for a fixed Noetherian ground ring $A$ of finite Krull dimension \emph{precobordism theories} are defined as certain quotients of rings $\Mc^*_+$ of cobordism cycles. For the following constructions, it is convenient to think about the second set of relations defining $\PCob^*$.

\begin{defn}\label{AdditivePCob}
The \emph{universal additive precobordism} $\PCob_a^*$ the quotient theory obtained from $\Mc_+^*$ by enforcing the homotopy fibre relation, as well as the relation 
$$e(\Ls \otimes \Ls') = e(\Ls) + e(\Ls')$$ 
on Euler classes of line bundles.
\end{defn}

\begin{defn}\label{MultiplicativePCob}
The \emph{universal multiplicative precobordism} $\PCob^0_m$ the quotient theory obtained from $\Mc_+^*$ by enforcing the homotopy fibre relation, as well as the relation 
$$e(\Ls \otimes \Ls') = e(\Ls) + e(\Ls') - e(\Ls) \bullet e(\Ls')$$ 
on Euler classes of line bundles. Note that the latter relation does not respect the grading of $\Mc_+^*$, and therefore we do not get a natural grading on $\PCob^0_m$.
\end{defn}

\begin{prop}\label{TensoringToGetUnivAdditiveAndMultProp}
Consider the precobordism rings $\PCob^*(X \to Y)$ as $\Lb$-algebras via the morphism $\Lb \to \PCob^*(pt)$ classifying the formal group law. Let $\Lb \to \Zb_a$ and $\Lb \to \Zb_m$ be the ring homomorphisms classifying respectively the additive and the multiplicative formal group laws on the integers. Then we have natural equivalences
$$\Zb_a \otimes_\Lb \PCob^* \cong \PCob^*_a$$
and
$$\Zb_m \otimes_\Lb \PCob^* \cong \PCob^0_m.$$
\end{prop}
\begin{proof}
Proofs of the claims are essentially the same, so let us prove the first one. We first note that by construction
$$e(\Ls_1 \otimes \Ls_2) = e(\Ls_1) + e(\Ls_2) \in \Zb_a \otimes_\Lb \PCob_a^* (X)$$
and therefore we obtain a well defined morphism $\psi: \PCob^*_a(X) \to \Zb_a \otimes_\Lb \PCob^*(X)$ given by identity on the level of cycles. To finish the claim, we need the inverse to be well defined, and by Theorem \fref{UniversalPropOfPreCob} this is essentially equivalent to showing that $\PCob^*_a(X)$ is a oriented cohomology theory in the sense of Definition \fref{OrientedCohomDef}. The nontrivial axioms to check are $(Norm)$ and $(FGL)$.
\begin{enumerate}
\item[$(Norm)$] Applying $(PP)$ to the definition of Euler class, we see that $e(\Ls) = i_{0!} (1_{D_0})$, where $i_0: D_0 \hookrightarrow X$ is the inclusion of the derived vanishing locus of the zero section of $\Ls$. But it follows from the homotopy fibre relation that $i_{0!}(1_{D_0}) = i_!(1_D)$ for any inclusion $i: D \hookrightarrow X$ of a virtual Cartier divisor $D$ in the linear system of $\Ls$. This proves the claim.

\item[$(FGL)$] By above, $e(\Ls)$ is nilpotent for $\Ls$ globally generated. A general line bundle $\Ls$ can be expressed as $\Ls_1 \otimes \Ls_2^\vee$ with $\Ls_i$ globally generated, and therefore $e(\Ls) = e(\Ls_1) - e(\Ls_2)$ is nilpotent for general $\Ls$. Moreover, the Euler classes satisfy a formal group law by construction, so we are done.
\end{enumerate}
Since $\PCob_a(X)$ is an oriented cohomology theory, and since it satisfies the additive formal group law, we obtain a well defined morphism $\eta: \Zb_a \otimes_\Lb \PCob^*(X) \to \PCob^*_a(X)$ commuting with pullbacks, pushforwards and preserving the identity elements. Hence, on the level of cycles, $\eta$ is the identity, and therefore $\eta$ and $\psi$ are inverses of each other.
\end{proof}

\subsection{Conner-Floyd theorem}\label{CFSubSect}

The purpose of this section is to prove the following theorem:

\begin{thm}\label{GeneralCF}
Suppose we are working over a Noetherian base ring $A$ of finite Krull dimension. Then the natural map
$$\eta_K: \Zb_m \otimes_\Lb \PCob^*(X) \cong \PCob^0_m (X) \to K^0(X)$$
sending $[V \xrightarrow{f} X]$ to $[f_* \Oc_V]$ is a natural isomorphism of rings commuting with pullbacks and Gysin pushforwards.
\end{thm}

Note that even though the most interesting part of the above theorem is that it generalizes Theorem 4.6 of \cite{An} to work over a more general base ring, it actually generalizes (slightly) also the theorem in characteristic 0 since we have potentially fewer relations. Moreover, it is self contained in that it does not any previous results of \cite{LM} or \cite{LS}. As usual, the hardest part of the theorem is the construction of Chern classes, which is already taken care of. As the proof is essentially the same as in \cite{An}, we will merely outline the  steps here, and refer to \emph{loc. cit.} for the details.

The idea is to show that the morphism $ch_m: K^0(X) \to \PCob^0_m(X)$, which is defined by the formula
\begin{equation}\label{CFCharacter}
ch_m[E] = \mathrm{rank}(E) - c_1(E^\vee),
\end{equation}
gives an inverse to $\eta_K$. This morphism is a homomorphism of rings: the additivity follows from Theorem \fref{WSF} and multiplicativity can be shown as in Section 4.1 of \cite{An}. Since $\eta_K$ preserves Chern classes, it is easy to check that $\eta_K \circ ch_m$ is the identity. Our next task is to show that $ch_m$ commutes with Gysin pushforwards along projective quasi-smooth morphisms, which is the hardest part of showing that also $ch_m \circ \eta_K$ is the identity. We are going to need several lemmas for this.

\begin{lem}
The map $ch_m: K^0 \to \PCob^0_m$ preserves Chern classes.
\end{lem}
\begin{proof}
Proceed as in the proof of Lemma 4.1 in \cite{An} and set $\beta=1$.
\end{proof}

Next we take care of a toy case.

\begin{lem}
Let $E$ be a vector bundle on $X$ and let $s: X \simeq \Pb(\Oc_X) \hookrightarrow \Pb(\Oc_X \oplus E)$ be the natural quasi-smooth inclusion. Then $ch_m$ commutes with $s_!$.
\end{lem}
\begin{proof}
Follows from the previous lemma as in the proof of Lemma 4.2 in \cite{An}.
\end{proof}

\begin{prop}\label{ChmCommutesWithPushforward}
The map $ch_m$ commutes with Gysin pushforwards along arbitrary projective quasi-smooth morphisms.
\end{prop}
\begin{proof}
Proved in the same way as Lemma 4.4 of \cite{An}.
\end{proof}

We are now ready to prove the main theorem of this section.

\begin{proof}[Proof of Theorem \fref{GeneralCF}]
We already know that $\eta_m \circ ch_K$ is the identity transformation, so it is enough to show that $ch_m \circ \eta_K$ is. As both $ch_m$ and $\eta_K$ preserve the identity element and commute with Gysin pushforwards, the same is true for the composition $ch_K \circ \eta_K$ as well. But then it must be the identity on generators: indeed
\begin{align*}
ch_K \circ \eta_K ([V \xrightarrow{f} X]) &= ch_K \circ \eta_K (f_!(1_V)) \\
&= f_!(ch_K \circ \eta_K (1_V)) \quad (\text{Proposition \fref{ChmCommutesWithPushforward}}) \\
&= f_!(1_V) \\
&= [V \xrightarrow{f} X] 
\end{align*}
and we can conclude that $ch_K \circ \eta_K$ is the identity.
\end{proof}

As an immediate consequence of this theorem, we can prove the following analogoue of a classical theorem of Levine-Morel.

\begin{cor}\label{UnivPropOfKThy}
Suppose we are working over a Noetherian base ring $A$ of finite Krull dimension. Then the algebraic $K$-theory of vector bundles on quasi-projective $A$-schemes $X$ is the universal oriented cohomology theory (in the sense of Definition \fref{OrientedCohomDef}) satisfying the multiplicative formal group law $F(x,y) = x + y - xy$.
\end{cor}

\subsection{Intersection theory and Riemann-Roch}\label{RRSubSect}

Recall that we can define a \emph{Chern character} map $ch_a: K^0(X) \to \Qb \otimes \PCob^*_a(X)$ by the formula
$$[\Ls] \mapsto e^{-c_1(\Ls)}$$
on line bundles and then extending formally to all vector bundles by using the splitting principle, and requiring $ch_a$ to be additive. Similarly, we may define the \emph{Todd class} $\Td: K^0(X) \to \Qb \otimes \PCob_a^*(X)$ by the formula 
$$[\Ls] \mapsto {c_1(\Ls) \over e^{-c_1(\Ls)} - 1}$$
on line bundles, and extending formally to all vector bundles by splitting principle, and requiring that $\Td(E \oplus F) = \Td(E) \bullet \Td(F)$. Note for $X$ quasi-projective, both classes are well defined on $K$-theory classes of perfect complexes, since $K^0(X)$ has a presentation as the Grothendieck group of vector bundles modulo exact sequences.

The purpose of this section is to prove the following theorem.

\begin{thm}\label{GeneralRR}
The maps
$$ch_a: K^0(X) \to \Qb \otimes \PCob_a^*(X)$$
are ring homomorphisms and commute with pullbacks. Moreover, given $f: X \to Y$ quasi-smooth and projective, we have that
\begin{equation}\label{ToddClasPF}
f_! \bigl( ch_a(\alpha) \bullet \Td(\Lb_{f}) \bigr) =  ch_a(f_!(\alpha))
\end{equation}
for all $\alpha \in K^0(X)$, where $\Lb_{f}$ is the relative cotangent complex of $X \to Y$. Finally, the induced map
$$ch_a: \Qb \otimes K^0(X) \to \Qb \otimes \PCob_a^*(X)$$
is an isomorphism.
\end{thm}

The proof is the same as in \cite{An}, but we have decided to write it down in a more explicit form here avoiding the bivariant formalism.

\begin{proof}
Let us start by \emph{twisting} the theory $\Qb \otimes \PCob^*_a$. We define a new oriented cohomology theory $\Hb_t^0$ by setting
$$\Hb_t^0(X) := \Qb \otimes \PCob^*_a$$
and
$$f_!^t(-) := f_!(- \bullet \Td(\Lb_f))$$ 
for any $f: X \to Y$ projective and quasi-smooth. The pullbacks of $\Hb_t^0$ are by definition the pullbacks of $\Qb \otimes \PCob^*_a$. Note that the pushforwards no longer respect degrees, and therefore we do not get a natural grading on $\Hb^0_t$. To prove that $\Hb^0_t$ is an oriented cohomology theory, we need to check that it satisfies the axioms in Definition \fref{OrientedCohomDef}.
\begin{enumerate}
\item[$(Fun)$] This is a computation: given $\alpha \in \PCob^*(X)$
\begin{align*}
(f \circ g)^t_! (\alpha) &= f_! \Bigl( g_! \bigl(\alpha \bullet \Td(\Lb_{f \circ g})\bigr)\Bigr) \\
&= f_! \Bigl( g_! \bigl(\alpha\bullet \Td(\Lb_g) \bullet g^* \bigl(\Td(\Lb_f) \bigr) \bigr)\Bigr) \\
&= f_! \Bigl( g_! \bigl(\alpha\bullet \Td(\Lb_g)\bigr) \bullet \Td(\Lb_f) \Bigr) \quad (\text{projection formula}) \\
&= f^t_!(g^t_!(\alpha)),
\end{align*}
proving the claim.

\item[$(PP)$] Follows easily from the fact that Todd-classes and cotangent compexes are stable under derived pullbacks.
\item[$(Norm)$] Follows from $(PP)$ and the homotopy fibre relation as in the proof of Proposition \fref{TensoringToGetUnivAdditiveAndMultProp}.
\item[$(FGL)$] Note that the cotangent complex of the zero section $s_0: X \to \Ls$ is $\Ls^\vee [1]$. Let us denote by $\pi$ the natural projection $\Ls \to X$. We can now compute the twisted Chern class for a line bundle $\Ls$ on $X$ in terms of the untwisted Chern classes (satisfying the formal group law)
\begin{align}\label{FirstTwistedC1}
c_1^t(\Ls) &= s_0^* s^t_{0!} (1_X) \\
&= s_0^* s_{0!} (1_X \bullet \Td(\Ls^\vee)^{-1}) \notag\\
&= s_0^* \bigl( s_{0!} (1_X) \bullet \pi^*\Td(\Ls^\vee)^{-1} \bigr) \quad (\text{projection formula}) \notag\\
&= s_0^*\bigl( s_{0!}(1_X) \bigr) \bullet \Td(\Ls^\vee)^{-1} \notag\\
&= c_1(\Ls) \bullet \Td(\Ls^\vee)^{-1} \notag\\
&= c_1(\Ls) \bullet  {e^{-c_1(\Ls^\vee)} - 1 \over c_1(\Ls^\vee)} \notag\\
&= 1 - e^{c_1(\Ls)}.\notag
\end{align}
Therefore
\begin{align*}
c^t_1(\Ls_1 \otimes \Ls_2) &= 1 - e^{c_1(\Ls_1 \otimes \Ls_2)}\\
&= 1 - e^{c_1(\Ls_1)} \bullet e^{c_1(\Ls_2)} \\ 
&= \bigl(1 - e^{c_1(\Ls_1)}\bigr) + \bigl(1 - e^{c_1(\Ls_2)}\bigr) - \bigl(1 - e^{c_1(\Ls_1)}\bigr) \bullet \bigl(1 - e^{c_1(\Ls_2)}\bigr) \\
&= c^t_1(\Ls_1) + c^t_1(\Ls_2) - c^t_1(\Ls_1) \bullet c^t_1(\Ls_2),
\end{align*}
and $\Hb_t^0$ satisfies the multiplicative formal group law. Nilpotency of Chern classes follows as in Proposition \fref{TensoringToGetUnivAdditiveAndMultProp}.
\end{enumerate}
The universal property of $K$-theory (Corollary \fref{UnivPropOfKThy}) therefore gives a morphisms of rings $ch_a': K^0(X) \to \Qb \otimes \PCob_a^*(X)$ commuting with pullbacks and satisfying formula (\fref{ToddClasPF}). To prove that $ch_a'$ coincides with $ch_a$ defined earlier, we note that $ch_a'$ must send Chern classes in $K$-theory to twisted Chern classes in the target, and compute
\begin{align*}
ch_a'([\Ls]) &= ch_a'(1 - c_1(\Ls^\vee)) \\
&= 1 - c_1^t(\Ls^\vee) \\
&= e^{c_1(\Ls^\vee)} \\
&= e^{-c_1(\Ls)} \\
&= ch_a([\Ls]).
\end{align*}
We have therefore proven everything else except equivalence with rational coefficients.

In order to finish the proof, it is useful to consider the presentation $\PCob^0_m$ for $K$-theory (see Theorem \fref{GeneralCF}). We want to find morphisms
$$\phi: \Qb \otimes \PCob^*_a(X) \to \Qb \otimes \PCob^0_m(X)$$ 
giving inverses to $ch_a$. To do this, consider the Todd classes in $\Qb \otimes \PCob^0_m(X)$ defined on line bundles by the formula
$$\Td'(\Ls) := {c_1(\Ls^\vee) \over \log(1 - c_1(\Ls^\vee))}$$
and the theory $\Hb^*_{t'}$ obtained from $\Qb \otimes \PCob^0_m(X)$ by a twisting construction as above. One then computes that
\begin{align}\label{SecondTwistedC1}
c_1^{t'}(\Ls) &= c_1(\Ls) \bullet \Td'(\Ls^\vee)^{-1} \\
&= c_1(\Ls) \bullet {\log\bigl(1 - c_1(\Ls)\bigr) \over c_1(\Ls)} \notag \\
&= \log\bigl( 1 - c_1(\Ls) \bigr) \notag
\end{align}
and therefore
\begin{align*}
c_1^{t'} (\Ls_1 \otimes \Ls_2) &= \log\bigl(1 - c_1(\Ls_1 \otimes \Ls_2) \bigr) \\
&= \log \bigl((1 - c_1(\Ls_1)) \bullet (1 - c_1(\Ls_2))  \bigr) \\
&= \log \bigl(1 - c_1(\Ls_1)\bigr) + \log \bigl( 1 - c_1(\Ls_2)  \bigr) \\
&= c_1^{t'}(\Ls_1) + c_1^{t'}(\Ls_2)
\end{align*}
so that $\Hb^*_{t'}$ is a oriented cohomology theory satisfying the additive formal group law. The universal property of $\PCob^*_a$ induces a unique morphisms
$$\phi: \Qb \otimes \PCob^*_a(X) \to \Qb \otimes \PCob_m^0(X)$$
which are defined on the level of cycles as
$$[V \xrightarrow{f} X] \mapsto f_!(1_V \bullet \Td'(\Lb_f)) \in \Qb \otimes \PCob_m^0(X).$$
Note that also the morphism $ch_a$ has a cycle level description as
$$[V \xrightarrow{f} X] \mapsto f_!(1_V \bullet \Td(\Lb_f)) \in \Qb \otimes \PCob_a^*(X).$$
These are inverses of each other. This is a computation, but before doing it, we note that we have to be slightly careful, since the Chern classes will be different in different theories. The first composition gives
\begin{align*}
ch_a\bigl(\phi([V \xrightarrow{f} X])\bigr) &= ch_a\bigl(f_!(1_V \bullet \Td'(\Lb_f))\bigr) \\
&= f_!(1_V \bullet \Td'(\Lb_f)^t \bullet \Td(\Lb_f))
\end{align*}
and the second one gives
\begin{align*}
\phi\bigl(ch_a([V \xrightarrow{f} X])\bigr) &= \phi\bigl(f_!(1_V \bullet \Td(\Lb_f))\bigr) \\
&= f_!(1_V \bullet \Td(\Lb_f)^{t'} \bullet \Td'(\Lb_f))
\end{align*}
where the superscripts $t,t'$ indicate the necessity of taking twists. We are done if we can show that $\Td'(\Lb_f)^t \bullet \Td(\Lb_f) = 1$ and $\Td(\Lb_f)^{t'} \bullet \Td'(\Lb_f)$.

But this reduces via splitting principle to checking the identity for line bundles. We compute that
\begin{align*}
\Td'(\Ls)^t \bullet \Td(\Ls) &= {c^t_1(\Ls^\vee) \over \log(1 - c^t_1(\Ls^\vee))} \bullet {c_1(\Ls) \over e^{-c_1(\Ls)} - 1} \\
&= {1 - e^{c_1(\Ls^\vee)} \over \log(e^{c_1(\Ls^\vee)})} \bullet {c_1(\Ls) \over e^{-c_1(\Ls)} - 1} \quad (\text{\fref{FirstTwistedC1}}) \\
&= {1 - e^{-c_1(\Ls)} \over -c_1(\Ls)} \bullet {c_1(\Ls) \over e^{-c_1(\Ls)} - 1} \\
&= 1
\end{align*}
and
\begin{align*}
\Td(\Ls)^{t'} \bullet \Td'(\Ls) &= {c_1^{t'}(\Ls) \over e^{-c_1^{t'}(\Ls)} - 1} \bullet {c_1(\Ls^\vee) \over \log(1 - c_1(\Ls^\vee))} \\
&= {\log\bigl( 1 - c_1(\Ls) \bigr)  \over e^{-\log\bigl( 1 - c_1(\Ls) \bigr) } - 1} \bullet {c_1(\Ls^\vee) \over \log(1 - c_1(\Ls^\vee))} \quad (\text{\fref{SecondTwistedC1}}) \\
&= -{\log\bigl( 1 - c_1(\Ls^\vee) \bigr)  \over e^{\log\bigl( 1 - c_1(\Ls^\vee) \bigr) } - 1} \bullet {c_1(\Ls^\vee) \over \log(1 - c_1(\Ls^\vee))} \\
&= -{\log\bigl( 1 - c_1(\Ls^\vee) \bigr)  \over - c_1(\Ls^\vee) } \bullet {c_1(\Ls^\vee) \over \log(1 - c_1(\Ls^\vee))} \\
&= 1
\end{align*}
which shows that $\phi$ and $ch_a$ are inverses of each other and finishes the proof.
\end{proof}

\subsection{Projective bundle formula}\label{PBFSubSect}

The purpose of this section is to generalize the projective bundle formula of \cite{AY} on trivial projective bundles to hold for arbitrary projective bundles. More precisely, we want to prove the following theorem:

\begin{thm}[Projective bundle formula]\label{PBF}
Let $X$ be a quasi-projective derived scheme over a Noetherian ring $A$ and let $E$ be a vector bundle of rank $r$ on $X$. Then
$$\PCob^*(\Pb(E)) \cong \PCob^*(X)[t]/ \bigl(c_r(E^\vee) - c_{r-1}(E^\vee)t + \cdots + (-1)^r t^r \bigr)$$
where $t \in \PCob^1(\Pb(E))$ is the first Chern class of $\Oc(1)$.
\end{thm} 

Our strategy is to first embed $\PCob^*(\Pb(E))$ to $\PCob^{*,1}(X)$, where $\PCob^{*,1}$ is the precobordism of line bundles from Section 6 of \cite{AY}. The Theorem \fref{PBF} will then follow from the results of \cite{AY} and some elementary manipulation of algebraic expressions. Let $\Ac$ be a line bundle on $X$ so that $\Ac \otimes E^\vee$ is globally generated, i.e., there exists a surjection $(\Ac^{\oplus R})^\vee \to E^\vee$ for some $R > 0$.

Next we are going to construct an embedding of $\PCob^*(\Pb(E))$ into $\PCob^{*,1}(X)$. In order to do so, we will first have to consider the diagram
\begin{equation}\label{ApproximationDiagram}
\begin{tikzcd}
& \colim_n \PCob^*(\Pb(\Ac^{\oplus n})) \arrow[->]{d}{j} \\
\PCob^*(\Pb(E)) \arrow[->]{r}{i} & \colim_n \PCob^*(\Pb(\Ac^{\oplus n} \oplus E)) \arrow[->]{d}{\Fc}  \\ 
& \PCob^{*,1}(X) 
\end{tikzcd}
\end{equation}
where the colimits over $n$ are induced by the obvious inclusions $\Ac^{\oplus n} \to \Ac^{\oplus n+1}$, where $\PCob^{*,1}(X)$ is the precobordism of line bundles over $X$ as in \cite{AY} Section 6. The maps $i$ and $j$ are induced by obvious inclusions of summands, and they are therefore injective by Proposition \fref{SummandGivesInjectivePush}. Finally, $\Fc$ is defined by the formula 
\begin{equation}\label{ForgetMapDef}
\Fc([V \xrightarrow{f} \Pb(\Ac^{\oplus n} \oplus E)]) := [V \xrightarrow{\pi \circ f} X, f^* \Oc(1)] \in \PCob^{*,1}(X),
\end{equation}
although it is not obvious yet that this gives rise to a well defined morphism. Note that if we can show that $\Fc$ is injective, then $\Fc \circ i$ provides the desired embedding. It turns out, that the right way to proceed is showing that $j$ and $\Fc$ are \emph{isomorphisms}.

Let us start with the well definedness of $\Fc$.

\begin{lem}
The formula (\fref{ForgetMapDef}) gives rise to a well defined homomorphism of Abelian groups $\Fc: \colim_n\PCob^*(\Pb(\Ac^{\oplus n} \oplus E)) \to \PCob^{*,1}(X)$.
\end{lem}
\begin{proof}
It is enough to show that the maps $\Fc_n: \PCob^*(\Pb(\Ac^{\oplus n} \oplus E)) \to \PCob^{*,1}(X)$ defined by
$$[V \xrightarrow{f} \Pb(\Ac^{\oplus n} \oplus E)] \mapsto [V \xrightarrow{\pi \circ f} X, f^* \Oc(1)]$$
are well defined, as they clearly commute with the structure morphisms of the colimit. By definition it is enough to check that $\Fc_n$ descent the double point cobordism relation: suppose we have a projective quasi-smooth morphism
$W \to \Pb^1 \times \Pb(\Ac^{\oplus n} \oplus E)$
so that the fibre $W_\infty$ over $\infty$ is a sum of two divisors $A + B$ in $W$. Moreover, let us denote by $\Ls$ the pullback of $\Oc(1)$ to $W$, and by $W_0$ the fibre over $0$. We now compute that
\begin{align*}
&\Fc_n([A \to \Pb(\Ac^{\oplus n} \oplus E)] + [B \to \Pb(\Ac^{\oplus n} \oplus E)] - [\Pb_{A \cap B}(\Oc(A) \oplus \Oc) \to \Pb(\Ac^{\oplus n} \oplus E)]) \\
&= [A \to X, \Ls \vert_A] + [B \to X, \Ls \vert_B] - [\Pb_{A \cap B}(\Oc(A) \oplus \Oc) \to X, \Ls \vert_{A \cap B}] \\
&= [W_0 \to X, \Ls \vert_{W_0}] \quad (\text{\cite{AY} Remark 6.7}) \\
&= \Fc_n([W_0 \to \Pb(\Ac^{\oplus n} \oplus E)])
\end{align*}
proving the claim.
\end{proof}

The following lemma will do a lot of the work for us.

\begin{lem}\label{SurjDoesNotMatter}
Let $X$ be a derived scheme and $E$ a vector bundle on $X$. Suppose that we have a line bundle $\Ls$ on a quasi-smooth and projective derived $X$-scheme $V$ and two surjections $s_1, s_2: E^\vee \vert_V \to \Ls$ giving rise to two maps $f_1, f_2: V \to \Pb(E)$. Then
$$[V \xrightarrow{f_1} \Pb(E)] = [V \xrightarrow{f_2} \Pb(E)] \in \PCob^*(\Pb(E)).$$
\end{lem}
\begin{proof}
Let $\iota_1: \Pb(E) \simeq \Pb(E \oplus 0) \hookrightarrow \Pb(E \oplus E)$ and $\iota_2: \Pb(E) \simeq \Pb(0 \oplus E) \hookrightarrow \Pb(E \oplus E)$ be the natural linear embeddings. As the pushforward morphism $\iota_{1!}$ is injective by Proposition \fref{SummandGivesInjectivePush}, it is enough to show that
$$\iota_{1!} [V \xrightarrow{f_1} \Pb(E)] = \iota_{1!} [V \xrightarrow{f_2} \Pb(E)] \in \PCob^*(\Pb(E \oplus E)).$$
We can now define a surjection $x_0 s_1 + x_1 s_2: E^\vee \vert_{\Pb^1_V} \oplus E^\vee \vert_{\Pb^1_V} \to \Ls (1)$ of vector bundles on $\Pb^1 \times V$ giving rise to an algebraic cobordism
$$\Pb^1 \times V \to \Pb^1 \times \Pb(E \oplus E)$$
showing that
$$\iota_{1!} [V \xrightarrow{f_1} \Pb(E)] = \iota_{2!} [V \xrightarrow{f_2} \Pb(E)] \in \PCob^*(\Pb(E \oplus E)).$$
On the other hand, we can argue similarly using the surjection $x_0 s_2 + x_1 s_2$ that
$$\iota_{1!} [V \xrightarrow{f_2} \Pb(E)] = \iota_{2!} [V \xrightarrow{f_2} \Pb(E)] \in \PCob^*(\Pb(E \oplus E)),$$
so we are done. 
\end{proof}

The following two lemmas show that $j$ and $\Fc$ are isomorphisms.

\begin{lem}\label{jInjLem}
The morphism $j: \colim_n \PCob^*(\Pb(\Ac^{\oplus n})) \to \colim_n \PCob^*(\Pb(\Ac^{\oplus n} \oplus E))$ from diagram (\fref{ApproximationDiagram}) is an isomorphism.
\end{lem}
\begin{proof}
We already know that $j$ is an injection, so it is enough to show that it is a surjection as well. Consider
$$\alpha := [V \to \Pb(\Ac^{\oplus m} \oplus E)] \in \colim_n \PCob^*(\Pb(\Ac^{\oplus n} \oplus E))$$
corresponding to a surjection $(\Ac^{\oplus m})^\vee \vert_V \oplus E^\vee \vert_V \to \Ls$. By construction there exists a surjection $(\Ac^{\oplus R})^\vee \to E^\vee$ so that we can form a composition of surjections
$$(\Ac^{\oplus R + m})^\vee \vert_V \to (\Ac^{\oplus m})^\vee \vert_V \oplus E^\vee \vert_V \to \Ls$$
showing by Lemma \fref{SurjDoesNotMatter} that $\alpha$ is in the image of $j$.
\end{proof}

\begin{lem}\label{PrecobordismOfLinesIsPInfty}
The morphism $\Fc \circ j: \colim_n \PCob^*(\Pb(\Ac^{\oplus n})) \to \PCob^{*,1}(X)$ from diagram (\fref{ApproximationDiagram}) is an isomorphism.
\end{lem}
\begin{proof}
Let us denote by $[\Ls] \in \PCob^{0,1}(X)$ the class of the cycle $[X \xrightarrow{\mathrm{Id}} X, \Ls]$. Note that $[\Ls]$ is a unit with inverse given by $[\Ls^\vee]$. Using the natural equivalences $\psi_\Ac: \Pb(\Ac^{\oplus n}) \cong \Pb(\Oc^{\oplus n})$, which on the level of functor of points are given by
$$[(\Ac^{\oplus n})^\vee \to \Ls] \mapsto [\Oc^{\oplus n} \to \Ac \otimes \Ls],$$
we can form the commutative diagram
\begin{equation*}
\begin{tikzcd}
\colim_n \PCob(\Pb(\Ac^{\oplus n})) \arrow[->]{r}{\psi_\Ac} \arrow[->]{d}{\Fc \circ j} & \colim_n \PCob(\Pb(\Oc^{\oplus n})) \arrow[->]{d}{j'} \\
\PCob^{*,1}(X) \arrow[->]{r}{[\Ac] \bullet} & \PCob^{*,1}(X)
\end{tikzcd}
\end{equation*}
The map $j'$ is by construction the isomorphism $\PCob_{\Pb^\infty}^*(X) \to \PCob^{*,1}(X)$ constructed in the Section 6 of \cite{AY}, and as also $\psi_\Ac$ and $[\Ac] \bullet$ are isomorphisms, also $\Fc \circ j$ must be one.
\end{proof}

Let us record the following result, which now is just a combination of the preceding lemmas.

\begin{prop}\label{EmbeddingToTriv}
Let $X$ be a quasi-projective derived scheme over a Noetherian ring $A$, and let $E$ be a vector bundle on $X$. Then the morphism $\iota: \PCob^*(\Pb(E)) \to \PCob^{*,1}(X)$ defined by the formula
$$[V \xrightarrow{f} \Pb(E)] \mapsto [V \xrightarrow{\pi \circ f} X, f^* \Oc(1)],$$
where $\pi$ is the natural projection $\Pb(E) \to X$, is an injection.
\end{prop}

We have reduced the proof of Theorem \fref{PBF} to understanding the image of $\iota$. Recall that the differentiation operator $\partial_{c_1}$ on $\PCob^{*,1}(X)$ was defined by linearly extending
$$\partial_{c_1} ([V \xrightarrow{f} X, \Ls]) := f_!(c_1(\Ls) \bullet 1_V)$$
so that they are linear over the subring $\PCob^*(X) \hookrightarrow \PCob^{*,1}(X)$ (cycles with trivial line bundles) and satisfy the formula $\partial_{c_1}([\Pb^i, \Oc(1)]) = [\Pb^{i-1}, \Oc(1)]$, where $\Pb^i$ is the empty scheme for $i<0$. Moreover, it is clear that the embedding $\iota$ of Proposition \fref{EmbeddingToTriv} exchanges the action of $c_1(\Oc(1))$ on the source to the action of $\partial_{c_1}$ on the target. We can therefore use the formula (\fref{PbEq2}) of Theorem \fref{ProjectiveBundleChernClass} conclude that if
$$\sum_{i \geq 0} a_i \bullet [\Pb^i, \Oc(1)] \in \PCob^{*,1}(X)$$
lies in the image of $\iota$, then the ``coefficients'' $a_i$ must satisfy the formulas
$$\sum_{j = 0}^r (-1)^{j} c_{r-j}(E^\vee) \bullet a_{n+j} = 0 \in \PCob^*(X)$$
for all $n \geq 0$. Examining the case $n=0$, we see that we are free to choose the coefficients $a_0, a_1, ..., a_{r-1}$, and after that everything is determined. Notice how arbitrary such a choice gives rise to a well defined element of $\PCob^{*,1}(X)$ since the Chern classes of $E^\vee$ are nilpotent, and therefore only finitely many $a_i$ are nonzero. We have therefore shown that \underline{\emph{the image of $\iota$ is contained in a free $\PCob^*(X)$-module of rank $r$}}. To finish off the proof, all that is left to do is to show that the images of $c_1(\Oc(1))^j$ for $0 \leq j \leq r-1$ form a basis.

\begin{proof}[Proof of Theorem \fref{PBF}]
It is enough to show that the elements $1, c_1(\Oc(1)), ..., c_1(\Oc(1))^{r-1}$ form a $\PCob^*(X)$-linear basis for $\PCob^*(\Pb(E))$. We note that the truthfulness of this claim does not change if we twist the vector bundle $E$ by a line bundle $\Ls$. Indeed: the derived scheme $\Pb(\Ls \otimes E)$ is isomorphic to $\Pb(E)$, but its anticanonical line bundle is equivalent to $\Ls^\vee \otimes \Oc(1)$. On the other hand, using (\fref{PbEq2}) and the formal group law of $\PCob^*$, we can conclude that 
$$c_1(\Ls^\vee \otimes \Oc(1)) = a_0 + a_1 \bullet c_1(\Oc(1)) + \cdots + a_{r-1} \bullet c_1(\Oc(1))^{r-1}$$
where $a_1$ is a unit and all the other $a_i$ are nilpotent. Hence $1, c_1(\Oc(1)), ..., c_1(\Oc(1))^{r-1}$ forms an $\PCob^*(X)$-basis if and only if $1, c_1(\Ls^\vee \otimes \Oc(1)), ..., c_1(\Ls^\vee \otimes \Oc(1))^{r-1}$ does.

As $X$ was assumed to be quasi-projective over a Noetherian ring $A$, we can assume w.l.o.g. that $E$ embeds into a trivial bundle $\Oc^{\oplus N}$ so that we get an exact sequence 
$$0 \to E \to \Oc^{\oplus N} \to E'' \to 0$$
whose left hand side gives rise to a linear embedding $i: \Pb(E) \hookrightarrow \Pb^{N-1}_X$. Moreover, note that the embedding $\iota: \PCob^*(\Pb(E)) \to \PCob^{*,1}(X)$ factors through $i_!$ and the embedding $\PCob^*(\Pb^{N-1}_X) \to \PCob^{*,1}(X)$, which, by the results of \cite{AY}, sends $[\Pb^i_X \hookrightarrow \Pb^{N-1}_X]$ to $[\Pb^i, \Oc(1)]$.

Note that by Proposition \fref{LinearEmbeddingIsVanishingLocus} $i_! (1_{\Pb(E)})$ is just the top Chern class $c_{N-r}(E''(1))$. On the other hand, using splitting principle and the formal group law, one can compute that, modulo nilpotent elements coming from $\PCob^*(X)$, 
\begin{align*}
c_{N-r}(E''(1)) &\equiv \sum_{j=0}^{N-r} c_{N-r-j}(E'') \bullet c_1(\Oc(1))^{N-r} \\
&\equiv c_1(\Oc(1))^j \\
&= [\Pb^{r-1}_X \hookrightarrow \Pb^{N-1}_X] \in \PCob^{*}(\Pb_X),
\end{align*}
and therefore $\iota(c_1(\Oc(1))^j) \equiv [\Pb^{r-1-i}, \Oc(1)]$ modulo nilpotent elements of $\PCob^*(X)$. Thus $\iota(c_1(\Oc(1))^i)$, for $0 \leq i \leq r-1$ must form an $\PCob^*(X)$-linear basis for the module of all elements of $\PCob^{*,1}(X)$ that can possibly be in the image of $\iota$ (by the analysis preceding this proof). The injectivity of $\iota$ then shows that they must also form a basis for $\PCob^*(\Pb(E))$, finishing the proof of the theorem.
\end{proof}

\subsection*{Bivariant projective bundle formula for precobordism theories}

Above, we have chosen to restrict to the associated cohomology theory of the universal precobordism theory, since we didn't introduce the bivariant formalism in the introduction. The purpose of this section is to show that Theorem \fref{PBF} can be easily generalized in two directions:
\begin{enumerate}
\item we can replace the universal precobordism rings $\PCob^*(X)$ with any precobordism rings $\Bb^*(X)$ (in other words, we can add relations);
\item instead of proving the projective bundle formula for the cohomology rings $\Bb^*(X)$, we can prove an analogous theorem for all the bivariant groups $\Bb^*(X \to Y)$.
\end{enumerate}
For background on bivariant precobordism theories, the reader can consult \cite{AY} Section 6.

\begin{thm}\label{BivariantPBF}
Let $\Bb^*$ be a bivariant precobordism theory in the sense of \cite{AY}, and suppose we have a morphism $X \to Y$ between derived schemes quasi-projective over a Noetherian ring $A$ of finite Krull dimension. Let $E$ is a vector bundle of rank $r$ on $X$. Then
\begin{enumerate}
\item we have a natural isomorphism of rings
$$\Bb^*(\Pb(E)) \cong \Bb^*(X)[t] / \langle f \rangle,$$
where $t = e(\Oc(1))$ and $f = \sum_{i=0}^r (-1)^i c_{r-i}(E^\vee) t^i$;

\item the morphism
$$\Bb^*(\Pb(E)) \otimes_{\Bb^*(X)} \Bb^*(X \to Y) \to \Bb^*(\Pb(E) \to Y)$$
defined by
$$\alpha \otimes \beta \mapsto \alpha \bullet \theta(\pi) \bullet \beta$$
gives an isomorphism of $\Bb^*(\Pb(E))$-modules, where $\pi$ is the structure morphism $\Pb(E) \to X$.
\end{enumerate}
\end{thm}
\begin{proof}
The proof of Theorem \fref{PBF} goes through with essentially no changes. The first task is to show that we have a natural embedding $\Bb^*(\Pb(E) \to Y) \hookrightarrow \Bb^{*,1}(X \to Y)$ analogous to the embedding of Proposition \fref{EmbeddingToTriv}. To do this, consider a bivariant version
\begin{equation}\label{ApproximationDiagram2}
\begin{tikzcd}
& \colim_n \Bb^*(\Pb(\Ac^{\oplus n}) \to Y) \arrow[->]{d}{j} \\
\Bb^*(\Pb(E) \to Y) \arrow[->]{r}{i} & \colim_n \Bb^*(\Pb(\Ac^{\oplus n} \oplus E) \to Y) \arrow[->]{d}{\Fc}  \\ 
& \Bb^{*,1}(X \to Y) 
\end{tikzcd}
\end{equation}
of the diagram (\fref{ApproximationDiagram}), where $\Ac$ is a line bundle so that we can find a surjection $(\Ac^{\oplus R})^\vee \to E^\vee$ for $R$ large enough. We note that
\begin{enumerate}
\item the maps $i,j$ induced by bivariant pushforwards are injective by an obvious bivariant analogue of Proposition \fref{SummandGivesInjectivePush};

\item the map $j$ is an isomorphism (cf. Lemma \fref{jInjLem});

\item the map $\Fc \circ j$ is a well defined isomorphism (this is proven exactly like Lemma \fref{PrecobordismOfLinesIsPInfty}, but using the more general isomorphism $\Bb^{*}_{\Pb^\infty}(X \to Y) \cong \Bb^{*,1}(X \to Y)$ constructed in \cite{AY}), and therefore also $\Fc$ is a well defined isomorphism.
\end{enumerate}
We therefore obtain an embedding $\Bb^*(\Pb(E) \to Y) \to \Bb^{*,1}(X \to Y)$ as in Proposition \fref{EmbeddingToTriv}. As in the proof of Theorem \fref{PBF}, we are reduced to understand the set of solutions to a easy linear recursion. The claims 1. and 2. can be proved in a very similar way as Theorem \fref{PBF}.
\end{proof}

\Addresses
\end{document}